\documentclass[12pt,reqno,english]{amsproc}
\usepackage[margin=0.95in]{geometry}
\usepackage{abstract, amsmath, amssymb, amsthm, array,
babel,bm,
catchfilebetweentags, cite, color,
dsfont,
emptypage,
fancyhdr, float, fnpos,
graphicx,
hyperref,
ifthen,
latexsym,
mathtools,
pdfpages,
qtree,
subcaption, 	
verbatim,
wrapfig}

\usepackage[all]{xy}

\hypersetup
{
    colorlinks,	
    citecolor=red,
    filecolor=black,
    linkcolor=blue,
    urlcolor=blue
}

\SelectTips{cm}{}

\DeclarePairedDelimiterX\setc[2]{\{}{\}}{\,#1 \;\delimsize\vert\; #2\,}
\DeclarePairedDelimiter\abs{\lvert}{\rvert}%
\DeclarePairedDelimiter\norm{\lVert}{\rVert}%

\makeatletter
\let\oldabs\abs
\def\abs{\@ifstar{\oldabs}{\oldabs*}}
\let\oldnorm\norm
\def\norm{\@ifstar{\oldnorm}{\oldnorm*}}
\makeatother

\newtheorem{theorem}{Theorem}[section]
\newtheorem{corollary}[theorem]{Corollary}
\newtheorem{lemma}[theorem]{Lemma}
\newtheorem{proposition}[theorem]{Proposition}
\theoremstyle{definition}
\newtheorem{definition}[theorem]{Definition}
\newtheorem{example}[theorem]{Example}
\newtheorem{remark}[theorem]{Remark}

\newtheoremstyle{algorithm} % name
    {\topsep}                    % Space above
    {\topsep}                    % Space below
    {\ttfamily}                   % Body font
    {}                           % Indent amount
    {\scshape}                   % Theorem head font
    {.}                          % Punctuation after theorem head
    {.5em}                       % Space after theorem head
    {}  % Theorem head spec (can be left empty, meaning ‘normal’)
\theoremstyle{algorithm}

\newtheoremstyle{theorem-w/o-number}
  {\topsep}   % ABOVESPACE
  {\topsep}   % BELOWSPACE
  {\itshape}  % BODYFONT
  {0pt}       % INDENT (empty value is the same as 0pt)
  {\bfseries} % HEADFONT
  {}         % HEADPUNCT
  {5pt plus 1pt minus 1pt} % HEADSPACE
  {}          % CUSTOM-HEAD-SPEC
\theoremstyle{theorem-w/o-number}
\newtheorem*{theorem*}{Theorem}
\newtheorem*{teorema*}{Teorema}
\newtheorem*{corollary*}{Corollary}
\newtheorem*{corolario*}{Corolario}
\newtheorem*{lemma*}{Lemma}
\newtheorem*{lema*}{Lema}
\newtheorem*{proposition*}{Proposition}
\newtheorem*{proposicion*}{Proposici\'on}

\newtheoremstyle{definition-w/o-number}
  {\topsep}   % ABOVESPACE
  {\topsep}   % BELOWSPACE
  {\normalfont}  % BODYFONT
  {0pt}       % INDENT (empty value is the same as 0pt)
  {\bfseries} % HEADFONT
  {}         % HEADPUNCT
  {5pt plus 1pt minus 1pt} % HEADSPACE
  {}          % CUSTOM-HEAD-SPEC
\theoremstyle{definition-w/o-number}
\newtheorem*{definition*}{Definition}
\newtheorem*{definicion*}{Definici\'on}
\newtheorem*{example*}{Example}
\newtheorem*{ejemplo*}{Ejemplo}\newtheorem*{remark*}{Remark}
\newtheorem*{observacion*}{Observaci\'on}

\numberwithin{equation}{section}

\let\tmp\oddsidemargin
\let\oddsidemargin\evensidemargin
\let\evensidemargin\tmp
\reversemarginpar

\newcolumntype{L}[1]{>{\raggedright\let\newline\\\arraybackslash\hspace{0pt}}m{#1}}
\newcolumntype{C}[1]{>{\centering\let\newline\\\arraybackslash\hspace{0pt}}m{#1}}
\newcolumntype{R}[1]{>{\raggedleft\let\newline\\\arraybackslash\hspace{0pt}}m{#1}}

\usepackage[font=footnotesize,labelfont=bf]{caption}

\newcommand{\N}{\mathbb{N}}

\newcommand{\bF}{\mathds F}

\newcommand{\bQ}{\mathbb Q}
\newcommand{\bR}{\mathbb R}
\newcommand{\bS}{\mathbb S}

\newcommand{\bV}{\mathbb V}

\newcommand{\bZ}{\mathbb Z}

\newcommand{\Image}{\text{\rm Im}\,}

\newcommand{\Ker}{\text{\rm Ker}\,}
\newcommand{\Coker}{\text{Coker}\,}

 \newcommand{\dlie}{\partial }

 \newcommand{\lib }{\mathbb{L}}

\newcommand{\catdga}{\operatorname{{\bf DGA}}}
\newcommand{\catdgl}{\operatorname{{\bf DGL}}}

\newcommand{\id}{\text{id}}

\newcommand{\De}{\Delta}
\newcommand{\al}{\alpha}
\newcommand{\be}{\beta}

\newcommand{\calH}{\mathcal{H}}
\newcommand{\calK}{\mathcal{K}}

\usepackage[latin1]{inputenc}

\begin{document}

\begin{center}
{\LARGE Optimising the topological information of the {$A_\infty$-persistence} groups}

\vspace{3mm}

{\large Francisco Belch\'{\i}\let\thefootnote\relax} 
\footnote{
Francisco Belch\'i-Guillam\'on, Mathematical Sciences, University of Southampton,
Building 54,
Highfield, Southampton SO17 1BJ, UK.
\textbf{\tt{frbegu@gmail.com}} \hfill ORCID: 0000-0001-5863-3343

%\vskip
% 1pt
This work has been supported by the Spanish MINECO grants MTM2010-18089 and
MTM2013-41762-P, by the Junta de
Andaluc\'\i a grant FQM-213 and by the UK's EPSRC grant
\emph{Joining the dots: from data to insight,}
EP/N014189/1.
 

All figures have been drawn by the author using mainly the software Inkscape.
}
\end{center}

\begin{abstract}
Persistent homology typically studies the evolution
of homology groups $H_p(X)$ (with coefficients in a field) along a filtration of topological spaces. 
$A_\infty$-persistence extends this theory by
analysing the evolution of subspaces such as 
$V \coloneqq \Ker {\Delta_n}_{| H_p(X)} \subseteq H_p(X)$,
where $\{\Delta_m\}_{m\geq1}$ denotes a structure of $A_\infty$-coalgebra on $H_*(X)$.
In this paper we illustrate how $A_\infty$-persistence can be useful beyond persistent homology by discussing the topological meaning of $V$, which is the most basic form of $A_\infty$-persistence group. In addition, we explore how to choose
$A_\infty$-coalgebras along a filtration to make the $A_\infty$-persistence groups carry more faithful information.
\end{abstract}

\tableofcontents

%%%%%%%%%%%%%%%%%%%%%%%%%%%%%%%%%%%%%
%%%%%%%%%%%%%%%%%%%%%%%%%%%%%%%%%%%%%
%%%%%%%%%%%%%%%%%%%%%%%%%%%%%%%%%%%%%
\section*{Introduction}

Persistent homology \cite{Carlsson-Zomorodian05,
Edelsbrunner-Letscher-Zomorodian02} computes the (persistent) Betti numbers of a sequence of topological spaces and continuous maps
$$ 
\xymatrix{
\calK\colon  & 
K_0
\ar[r] & K_1
\ar[r] & \ldots
\ar[r] & K_N
}$$
created by varying a parameter such as
time, thickness, intensity, height, \emph{etc.}
Depending on the context, this can allow us to discover highly non-linear structure in data or to compute novel geometric descriptors of shapes.
For instance,
G. Carlsson \emph{et al.} considered a point cloud dataset built from 3 by 3 high-contrast patches from grey-scale natural images and studied an unknown space $X$ on which the points accumulated with high density.
They used persistent homology to estimate some Betti numbers of $X$,
but went further to find that the 2-skeleton of $X$ formed a Klein bottle 
\cite{Carlsson-Ishkhanov-de_Silva-Zomorodian08}.
This extra knowledge was then used as starting point to develop a dictionary for texture representation
\cite{Perea-Carlsson14}. This is a great example of how important it can be to find
persistent topological information beyond the level of Betti numbers.
$A_\infty$-persistence \cite{Belchi-Murillo15} aims at doing so in a semi-automated way by studying
persistent topological information at the level of $A_\infty$-structures --
algebraic constructions that can encode attributes related to cup product and higher order Massey products (see Fig. \ref{fig:DiagramaResumen}).

\begin{figure}[h!]
       \centering
		\includegraphics[width=0.6\textwidth]
		{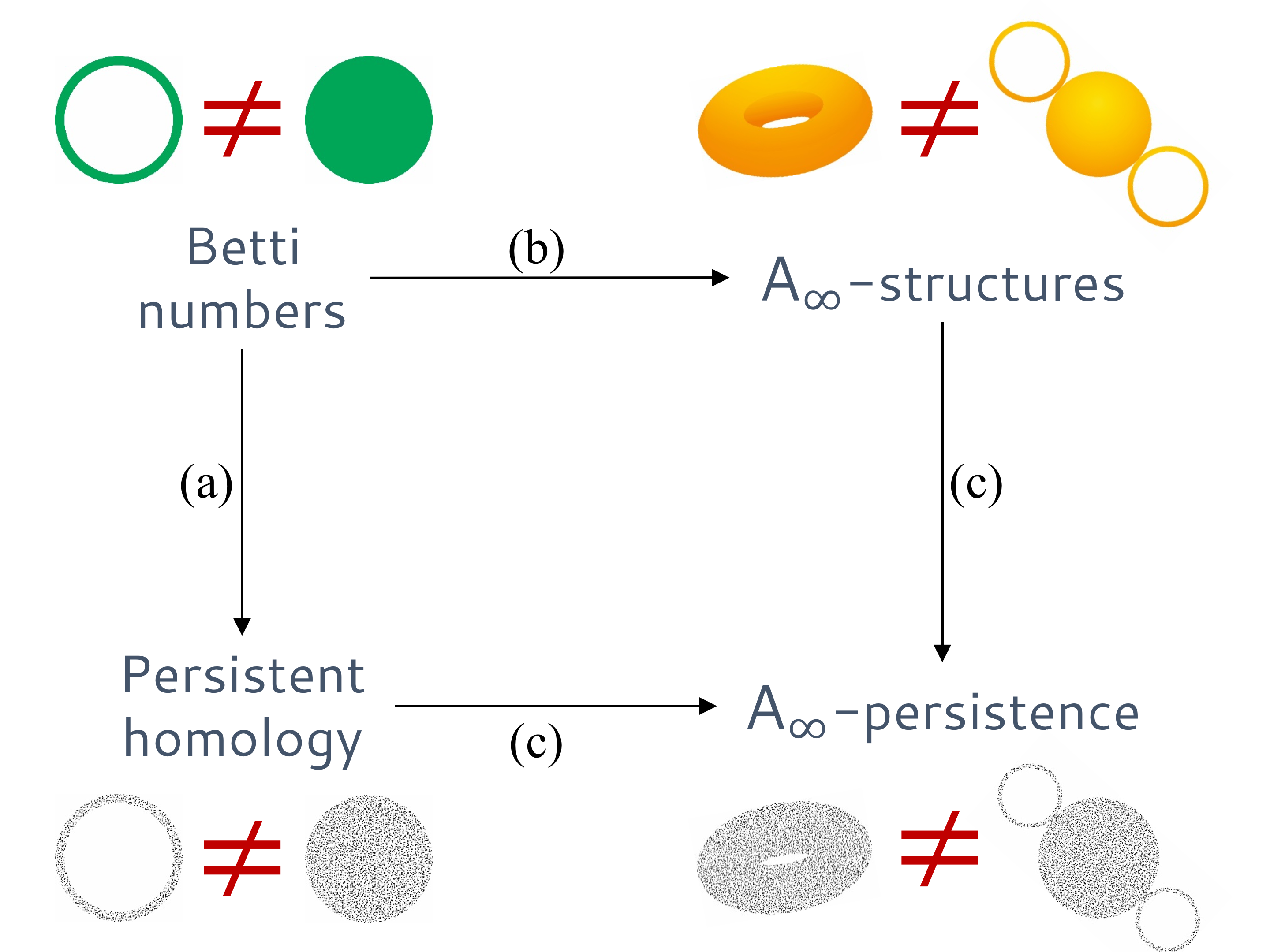}
	\caption{Persistence applied to point cloud datasets:
(a) Betti numbers 
contribute valuable topological information about solid objects,
able to distinguish \emph{e.g.,} a circumference from a disc.
Persistent homology adapts these invariants to the study of point cloud datasets,
so that a sample from a circumference can be told apart from that of a disc. 
(b) $A_\infty$-structures provide a much more detailed structural description 
than that of Betti numbers. They contain all the information of the cup product, and hence can tell apart \emph{e.g.,} a torus from a wedge of spheres 
$\bS^1 \vee \bS^2 \vee \bS^1$. On top of that, they contain information related to  (higher order) Massey products, allowing us to also distinguish \emph{e.g.,} 3 unlinked rings from the Borromean rings (see Fig. \ref{fig:UnlinkedVsBorromean}).
(c) In this context, the aim of $A_\infty$-persistence is to use the strengths
 of $A_\infty$-(co)algebras to sharpen the tool of persistent homology,
 \emph{e.g.,}
to allow the distinction of more involved point cloud datasets.}
	\label{fig:DiagramaResumen} 
\end{figure}

We plan to work with $A_\infty$-coalgebras on the homology $H_*(X)$ of a space $X$ and with $A_\infty$-algebras on its cohomology $H^*(X)$. These structures codify a good deal of topological information of $X$, as this paper will illustrate.
For all concepts below related to $A_\infty$-structures, see \S \ref{sec_A-inftyStructs}.
Given two homotopy equivalent spaces,
the transferred $A_\infty$-coalgebras they induce coincide up to isomorphism.
Hence, ideally, we would like to study the persistence of this whole
isomorphism class of $A_\infty$-structures. 
Unfortunately, this class is too large to compute in general,
so we end up having to sacrifice some of this structure in return for computability.
More concretely,
after choosing a field of coefficients,
we fix a transferred $A_\infty$-coalgebra
$(H_*(K_i), \{\Delta^i_n\}_n)$
for each space $K_i$ in $\calK$
and think of defining a vector subspace $V_i \subseteq H_p(K_i)$
for every $i=0, \ldots, N$, so that the following two conditions hold:
\begin{enumerate}
\item $V_i$ contains enough information from the
$A_\infty$-coalgebra $\{\Delta^i_n\}_n$ to be useful.
\item $V_i$ is simple enough to allow a feasible persistence computation.
\end{enumerate}
Persistent homology encodes in a barcode 
the evolution along $\calK$ of the vector space $H_p(K_i)$.
By (2) we mean that, similarly,
we should be able to encode in a barcode the evolution along $\calK$ of the subspace 
$V_i \subseteq H_p(K_i)$.
In \cite{Belchi-Murillo15}, we used zigzag persistence \cite{Carlsson-de_Silva10}
to prove that 
$V_i \coloneqq \Ker {\Delta_n}_{| H_p(K_i)}$
satisfies (2), where here and henceforth $f_{| W}$ denotes the restriction of the map $f$ to the subspace $W$.

In this article, we do the following.
On the one hand, we draw a more comprehensive
picture of $V_i$'s trade-off between simplicity and retained information
from the $A_\infty$-structure. To do so, we show both what we lose
by not using the whole isomorphism class of $A_\infty$-structures, and what
we still gain with respect to Betti numbers.
Specifically, we show that $\dim V_i$
is not a homotopy invariant in the most general setting
(Ex. \ref{prop_De_n-not-invariant}),
but this dimension can still recover in some cases information not readily available in
the Betti numbers or even the cup product
(Thm. \ref{THM:min-invariant}, 
Cor. \ref{COR:min-invariant} , Prop. \ref{Prop:Borromean-m3} and \ref{Prop:ExtendingBorromean}).
On the other hand, recall that the presence of a homological feature in the sequence $\calK$ may sometimes be represented in a zigzag barcode by two or more different, totally independent intervals. This can mislead us to believe that they may be detecting 
more than one feature, as exemplified in \cite[\S 5.6.6]{Tausz-Carlsson11}.
As $A_\infty$-persistence uses a tweaked version of a zigzag barcode
(see \cite[Def. 2.8]{Belchi-Murillo15}), in \S \ref{compatibleAinftyStructures}
we will show how this ambiguity affects 
$A_\infty$-persistence
and discuss a way to bypass this issue.

This work is organised as follows: in \S \ref{sec_A-inftyStructs} 
we provide some background on $A_\infty$-structures.
In \S \ref{sec_barcodes}
we recall, from an algebraic point of view, the ordinary and $A_\infty$ versions of \emph{barcodes}. These are tools used to encode persistence information. In addition, we give an alternative proof of Lemma \ref{lemma_persModsSplit}, 
which has as corollary the barcode decomposition theorem of persistent homology
(Thm. \ref{Thm:Fundamental_thm_PH}).
In \S \ref{sec_topologicalMeaning}
we illustrate when $A_\infty$-persistence 
can be useful beyond persistent homology by discussing the 
topological meaning of the measurements $\dim V_i$ and their duals in cohomology
(Ex. \ref{prop_De_n-not-invariant}, Prop. \ref{Prop:Borromean-m3} and \ref{Prop:ExtendingBorromean}, Thm. \ref{THM:min-invariant} and
Cor. \ref{COR:min-invariant} ).
The results in this section can be used to choose the right $n$ to study the persistence of $\Delta_n$ and they also suggest a new persistence approach to links.
In \S \ref{compatibleAinftyStructures}
we exhibit how the aforementioned zigzag ambiguity problem affects $A_\infty$-persistence and propose a way to deal with it. This involves showing that we can sometimes make coherent choices
of $A_\infty$-structures on the sequence $\calK$
so that what $A_\infty$-persistence
has to decompose is a persistence module 
(Thm. \ref{THM:Coherent-choices-Quillen}).
We then prove the barcode decomposition theorem of $A_\infty$-persistence
for the case of such coherent choices of $A_\infty$-structure
(Thm. \ref{Thm_Fund_Thm_A_infty-p_Functorial})
to stress how much things simplify by bypassing 
the need for zigzag machinery.
We leave some of the proofs for \S \ref{sec_appendix} and 
finish with some conclusions in \S \ref{sec_conclusions}.

The results in \S \ref{sec_topologicalMeaning} and \S \ref{compatibleAinftyStructures} should be of interest both to algebraic topologists and to more applied scientists involved in data analysis in some way.

Let us fix some notation and assumptions for the rest of the paper.
For the basic notions of algebraic topology used (such as cup product) we refer the reader to classics such as \cite{Hatcher02}.
% \cite{Hatcher02, Munkres_Elements_of_AT}.
Throughout this paper, we will
always work over a field of coefficients $\bF$ which we will usually omit from the notation.
For instance,
$H_*(X)$ will denote the homology of $X$ with coefficients in $\bF$.
Every topological space $X$
considered will be assumed to have
finite dimensional homology groups in all degrees,
$\dim H_p(X) < \infty$.
We will denote by 
$$ 
\xymatrix{
\calK\colon  & 
K_0
\ar[r] & K_1
\ar[r] & \ldots
\ar[r] & K_N
}$$
a finite sequence of topological spaces and
continuous maps
and, for any integers $p\geq 0$ and $0 \leq i \leq j \leq N$,
we will denote by 
$$
f^{i,j}_p\colon  H_p(K_i)
\longrightarrow 
H_p(K_j)
\hspace{10mm}
\text{and}
\hspace{10mm}
f^{i,j}\colon  H_*(K_i)
\longrightarrow 
H_*(K_j)
$$
\noindent the linear maps induced in homology
by the composition
$
\xymatrix{
K_i
\ar[r] & K_{i+1}
\ar[r] & \ldots
\ar[r] & K_j.
}$

%%%%%%%%%%%%%%%%%%%%%%%%%%%%%%%%%%%%%
%%%%%%%%%%%%%%%%%%%%%%%%%%%%%%%%%%%%%
%%%%%%%%%%%%%%%%%%%%%%%%%%%%%%%%%%%%%
\section{$A_\infty$-structures}
\label{sec_A-inftyStructs}

Looking for algebraic structures that include
all the information of the homology groups and even of
the standard cohomology algebra given by the cup product, 
and provide still more detailed homological information, 
we naturally bump into $A_\infty$-structures.

\begin{definition}\label{DEF:A-infty-coalgebra}
An \textbf{${\bm A_\infty}$-coalgebra structure} 
$\{\De_n\}_{n \geq 1}$ on a graded vector space $C$
is a family of maps
$$\Delta_n\colon  C\longrightarrow C^{\otimes n}$$
of degree $n-2$
such that, for all $n \geq 1$,
the following  \emph{Stasheff identity} holds:
$$
\text{SI}(n)\colon 
 \sum_{i=1}^n \sum_{j=0}^{n-i} (-1)^{i+j+ij}
		\left( 1^{\otimes n-i-j} \otimes \De_i \otimes 1^j \right)
		\,\De_{n-i+1} = 0. $$	
\end{definition}

The identities $\text{SI}(n)$ for $n=1, 2, 3$ state that if
$(C, \{\De_n\}_{n \geq 1})$
is an $A_\infty$-coalgebra,
then $\De_1$ is a differential on $C$ and the comultiplication
$\De_2$ is coassociative up to the chain homotopy $\De_3$.
Moreover, any differential graded coalgebra (DGC henceforth) $(C,\dlie,\De)$
can be viewed as an $A_\infty$-coalgebra 
$(C, \{\De_n\}_{n \geq 1})$ by setting
$ \Delta_1 = \dlie, \Delta_2 = \De,$ and $\Delta_n =0$ for all $n > 2.$
%
%\begin{definition}\label{DEF:minimal-A-coalgebra}
An $A_\infty$-coalgebra
$(C, \{\De_n\}_{n \geq 1})$
is called \textbf{minimal}
if 
$\De_1=0.$
%\end{definition}

\begin{definition}\label{DEF:Morph-A-coalg}
A \textbf{morphism of $\bm{A_\infty}$-coalgebras} 
$$ f\colon (C,\{\De_n\}_{n\ge 1})\to (C',\{\De'_n\}_{n\ge 1})$$
is a family of maps
 $$
 f_{(k)}\colon C\longrightarrow {C'}^{\otimes k},\qquad k\ge 1,
 $$
of degree $k-1$, such that
for each $i\ge 1$,
the following identity holds:
$$
\text{MI}(i)\colon \sum_{\substack{
p+q+k=i\\ q \geq 1, p,k \geq 0
}}
(1^{\otimes p}\otimes\Delta'_q\otimes 1^{\otimes k})f_{(p+k+1)}=
\sum_{\substack{
k_1+\dots+k_\ell=i\\ l, k_j \geq 1
}}
(f_{(k_1)}\otimes\dots\otimes f_{(k_\ell)})\Delta_\ell.
$$

We say that the
morphism of $A_\infty$-coalgebras $f$
is:
\begin{itemize}
\item
an \textbf{isomorphism}
if $f_{(1)}$ is an isomorphism,
\item
a \textbf{quasi-isomorphism}
if $f_{(1)}$ induces an isomorphism
in homology.
\end{itemize}
\end{definition}

There are some trivial $A_\infty$-coalgebra structures one can always endow a graded vector space $C$ with, such as the one given by $\Delta_n = 0$ for all $n$.
This structure does not give in general any new information, so instead,
we will focus our attention in a special kind of
$A_\infty$-coalgebras we will refer to as the \emph{transferred} ones.
As we will see between this section and \S \ref{sec_topologicalMeaning},
out of all the $A_\infty$-coalgebra
structures one could endow $H_*(X)$ with,
the transferred ones
can give particularly meaningful information
about the homotopy type of $X$.
For instance, 
we mentioned that
$\De_3$ measures the degree to which
the coassociativity of $\De_2$
can be relaxed in an 
$A_\infty$-coalgebra
$(C, \{\De_n\}_{n \geq 1})$.
However, 
even if
$\De_2$ is strictly coassociative,
$\De_3$ may be non-trivial,
and in a transferred $A_\infty$-coalgebra, 
this non-triviality can give
crucial information
(see \S \ref{sec_topologicalMeaning}).

\begin{definition}\label{DEF:good-A-coalg}
We will say that an $A_\infty$-coalgebra
$\left( H_*(X), \{ \De_n \}_n \right)$ 
on the homology of a space $X$
is a \textbf{transferred
$\bm{A_\infty}$-coalgebra} (induced by $X$) if it is minimal and
quasi-isomorphic
to the 
$A_\infty$-coalgebra
$$\left( C_*(X), \{\dlie, \De, 0, 0, \ldots \} \right),$$
where $\left( C_*(X), \dlie \right)$
denotes the singular chain complex of $X$
and $\De$ denotes an approximation to the diagonal.
We will drop the \emph{'induced by $X$'} from the notation
when no confusion is possible.
\end{definition}

Analogously, in a reduced homology setting, we will refer to an $A_\infty$-coalgebra
$\left( \widetilde H_*(X), \{ \De_n \}_n \right)$ 
as a transferred one if it is minimal and quasi-isomorphic 
to the induced $A_\infty$-coalgebra at the reduced chains level
$$\left( \widetilde C_*(X), \{\widetilde\dlie, \widetilde\De, 0, 0, \ldots \} \right).$$

An immediate consequence of Definition \ref{DEF:good-A-coalg} is that
all transferred
$A_\infty$-coalgebras on $H_*(X)$
induced by $X$
are isomorphic.

In the proof of Prop. \ref{Prop:Borromean-m3}
we will use the following remark.

\begin{remark}
\label{Rmk_HTT}
For every space $X$,
there always exist
transferred
$A_\infty$-coalgebras on $H_*(X)$
induced by $X$, which can be computed in several ways,
most of them amounting to an application of the
Homotopy Transfer Theorem \cite{Kontsevich-Soibelman00,
Loday-Vallette12} (hence the name \emph{transferred}
$A_\infty$-coalgebras). This can be used as an algorithm that takes as input a diagram of the form
\begin{equation}\label{ecuacion1}
\xymatrix{ \ar@(ul,dl)@<-5.5ex>[]_\phi  & (M,d) \ar@<0.75ex>[r]^-\pi  & (N,d) \ar@<0.75ex>[l]^-\iota }
\end{equation}
in which $(N, d)$ is a chain complex, $(M,d)$ is a DGC with comultiplication $\Delta$, and the degree 0 chain maps $\pi$ and $\iota$ and the degree 1 chain homotopy $\phi$ make the following hold:
$\pi \iota={\rm id}_N$, $\pi  \phi=\phi\iota=\phi^2=0$ and $\phi$ is a chain homotopy between ${\rm id}_M$ and $\iota \pi $, \emph{i.e.,} $\phi d+d\phi=\iota \pi -{\rm id}_M$. 
The output of the algorithm includes an explicit minimal $A_\infty$-coalgebra structure $\{\De_n\}_n$ on $N$ with $\Delta_2=\pi^{\otimes 2} \Delta \iota$ and such that,
if $\phi = 0$, then $\De_n = 0$ for all $n>2$.
\end{remark}

The dual notion of an $A_\infty$-coalgebra
is that of an $A_\infty$-algebra.
Depending on the tools more suitable 
to making the computations,
it can be more convenient to work with one notion or the other.

\begin{definition}\label{DEF:A-algebra}
An \textbf{$\bm{A_\infty}$-algebra structure} $\{\mu_n\}_{n \geq 1}$
on a graded vector space $A$ is a family of maps
$$\mu_n\colon A^{\otimes n} \longrightarrow A$$
of degree $2-n$
such that, for all $n \geq 1$,
the following \emph{Stasheff identity} holds:
$$
\text{SI}(n)\colon 
\sum_{\substack{
n=r+s+t\\ s \geq 1\\ r,t \geq 0
}} (-1)^{r+st} \mu_{r+1+t}
\left( 1^{\otimes r} \otimes \mu_s \otimes 
1^{\otimes t} \right) = 0.
$$
\end{definition}

Everything said for $A_\infty$-coalgebras dualizes to
$A_\infty$-algebras.
In an $A_\infty$-algebra
$(A, \{\mu_n\}_{n \geq 1})$,
$\mu_3$ measures
the associativity of $\mu_2$
but even if $\mu_2$ is strictly associative,
$\mu_3$ may be non-trivial
and give useful information
(see \S \ref{sec_topologicalMeaning}).
The Homotopy Transfer Theorem also works
for differential graded algebras (DGA henceforth) and $A_\infty$-algebras,
and we can define transferred 
$A_\infty$-algebras
on the cohomology of a space, $H^*(X)$,
just as in Definition
\ref{DEF:good-A-coalg},
using the cup product instead of an approximation
to the diagonal.

In particular, for any transferred 
$A_\infty$-algebra $\{\mu_n\}_{n}$ on
$H^*(X)$, $\mu_2$ coincides with the cup product.
Hence, transferred $A_\infty$-structures encode all the information in the homology
groups of $X$ and in its cohomology algebra as well, but there is more.
For instance,
in \cite[Thm. 1.3]{Belchi-Murillo15}
we exhibited a way to build pairs of spaces with isomorphic homology groups and isomorphic cohomology algebras 
but non-isomorphic transferred $A_\infty$-coalgebras (on their homology),
and T. Kadeishvili proved in
\cite[Prop. 2]{Kadeishvili80}
that under mild conditions on a topological space $X$,
any transferred $A_\infty$-algebra on its cohomology
determines the cohomology of its loop space, $H^*(\Omega X)$,
whereas the cohomology ring of $X$ alone
does not.

%%%%%%%%%%%%%%%%%%%%%%%%%%%%%%%%%%%%%
%%%%%%%%%%%%%%%%%%%%%%%%%%%%%%%%%%%%%
%%%%%%%%%%%%%%%%%%%%%%%%%%%%%%%%%%%%%
\section{Barcodes}
\label{sec_barcodes}

In persistence, we are interested in finding out how 
the topology evolves along the sequence of topological spaces and 
continuous maps $\calK$
created by varying a parameter such as
time, thickness, intensity, height, \emph{etc.}
This is done by applying algebraic-topology constructions to $\calK$
and decomposing the result into the smallest pieces possible in a certain sense,
obtaining a graphical representation called a \emph{barcode}.
In this section we recall, from an algebraic point of view, the notions
of barcode in persistent homology \cite{Carlsson-Zomorodian05,
Edelsbrunner-Letscher-Zomorodian02} and $A_\infty$-persistence \cite{Belchi-Murillo15}.
In addition, we give an alternative proof of the decomposition theorem of persistent homology.

\begin{definition}
\label{Def:P.H.groups}
For every $0\leq i \leq j \leq N$
and $p\geq 0$, the 
\textbf{$\bm{p^{\text{\rm \textbf{th}}}}$ persistent
homology group}
of the sequence $\calK$
between $K_i$ and $K_j$
is defined as the vector space
$$ H^{i,j}_p ( \calK ) \coloneqq
\Image f^{i,j}_p $$
and its corresponding
\textbf{persistent Betti number} is defined as
$$ \be^{i,j}_p ( \calK ) \coloneqq
\dim H^{i,j}_p ( \calK ). $$
\end{definition}

\begin{definition}\label{Def:pers.mod}
A \textbf{persistence module} $\bV$
is a finite sequence of 
	finite dimensional vector spaces and
	linear maps between them 
	of the form
	$$\xymatrix{
	V_0 \ar[r] & V_1 \ar[r] &
	\,\ldots\, \ar[r] & V_N.
	}$$
\end{definition}

The Fundamental
Theorem of Persistent Homology
can be stated as follows:

\begin{theorem}
\label{Thm:Fundamental_thm_PH}
\emph{\cite[\S 3]{Carlsson-Zomorodian05}
\textbf{(Fundamental
Theorem of Persistent Homology)}}
For any integer $p \geq 0$, 
there exists a unique
multiset $M_p$ of
intervals of the form $[i,j)$ for $0 \leq i < j \leq N$,
and intervals of the form $[i,\infty)$ for $0 \leq i \leq N$,
such that	for all \,$0 \leq i \leq j \leq N$,
the dimension
$\dim H^{i,j}_p ( \calK )$
equals the number of intervals
in $M_p$
(counted with multiplicities)
which contain the interval $[i,j]$.

In particular, 
$\dim H_p ( K_i )$
equals the number of intervals
in $M_p$
(counted with multiplicities)
which contain the integer $i$.
\end{theorem}

%\begin{definition}
%\label{Def:classicalBarcode}
The multiset $M_p$ in Thm.
\ref{Thm:Fundamental_thm_PH}
is known as the 
$\bm{p^{\text{\textbf{th}}}}$ \textbf{barcode}
of $\calK$.
%\end{definition}
Thm. \ref{Thm:Fundamental_thm_PH} is a consequence of the following
result:

\begin{lemma}
\label{lemma_persModsSplit}
\emph{\cite[\S 3]{Carlsson-Zomorodian05}}
	Let
	$\xymatrix{
	V_0 \ar[r]^-{f^{0,1}} & V_1 \ar[r]^-{f^{1,2}} &
	\,\ldots\, \ar[r]^-{f^{N-1,N}} & V_N
	}$
	be a persistence module
	and set
	$$
	f^{i,j}\coloneqq
	\begin{cases}
		f^{j-1,j}\circ\ldots\circ f^{i,i+1},
		\quad & i+1 < j, \\
		{\id}_{V_i},
		\quad & i=j.
	\end{cases}
	$$
	Then 
	there exists a unique multiset $M$ of
	intervals of the form $[i,j)$ for $0 \leq i < j \leq N$,
	and intervals of the form $[i,\infty)$ for $0 \leq i \leq N$,
	such that
	for all \,$0 \leq i \leq j \leq N$,
	the dimension
		$\dim \Image f^{i,j}$
	can be computed as the sum of
	the multiplicity of each interval in $M$
	that contains $[i,j]$.	
\end{lemma}

G. Carlsson and A. Zomorodian first proved 
Lemma \ref{lemma_persModsSplit} and therefore
Thm. \ref{Thm:Fundamental_thm_PH} in its full generality in \cite[\S 3]{Carlsson-Zomorodian05}
with a slightly different notation and vocabulary. They made use of the structure
theorem of finitely generated
modules over Principal Ideal Domains.
Next we give an alternative proof of 
Lemma \ref{lemma_persModsSplit}
(and therefore of Thm. \ref{Thm:Fundamental_thm_PH})
making use, instead, of the following well-known fact in linear algebra:

\begin{lemma}
\label{Lemma_FrobeniusRanks}
\textbf{(Frobenius inequality on matrix ranks)}
Let $A, B$ and $C$ be matrices such that the products
$AB, ABC$ and $BC$ are defined. Then
$$  \operatorname{rank}(B) - \operatorname{rank}(AB) 
- \operatorname{rank}(BC) + \operatorname{rank}(ABC) \geq 0.$$
\end{lemma}

Here is an alternative proof of Lemma \ref{lemma_persModsSplit}:

\begin{proof}
Let us set
$$ d^{i,j}\coloneqq
\begin{cases}
\dim \Image f^{i,j}, \quad &
0 \leq i \leq j \leq N,\\
0, \quad & \text{\rm otherwise}.
\end{cases}
$$
Let us assume that there existed a multiset $M$
consisting of:
\begin{itemize}
\item a number $N^{i,j-1}$	of intervals of the form $[i,j)$
for all $0 \leq i < j \leq N$, and
\item a number $N^{i,N}$ of intervals of the form $[i,\infty)$ 
for all $0 \leq i \leq N$,
\end{itemize} 
such that, for all \,$0 \leq i \leq j \leq N$,
$$ d^{i,j} = \#\{	\text{\rm intervals in M that contain $[i,j]$} \}.	$$
Then, an inclusion-exclusion-type argument shows that, for all $0 \leq i \leq j \leq N$,
$N^{i,j}$ is forced to be
\begin{equation}
N^{i,j} 
=
d^{i,j}-d^{i-1,j}-d^{i,j+1}+d^{i-1,j+1}.
\end{equation}
Hence, if $M$ exists, it must be unique.
On the other hand, 
the existence of $M$
amounts to
$N^{i,j} \geq 0$ holding for every $0\leq i
\leq j \leq N$. Finally, notice that Lemma \ref{Lemma_FrobeniusRanks} implies that these inequalities $N^{i,j} \geq 0$ hold, proving the existence of $M$.
\end{proof}

Thm. \ref{Thm:Fundamental_thm_PH} then
follows from applying Lemma \ref{lemma_persModsSplit}
to the persistence module
$$ 
\xymatrix{
H_p(K_0)
\ar[r]^-{f^{0,1}_p} & H_p(K_1)
\ar[r]^-{f^{1,2}_p} & \ldots
\ar[r]^-{f^{N-1,N}_p} & H_p(K_N).
}$$

For the rest of this section, 
choose a transferred $A_\infty$-coalgebra
structure $\{\De_n^i\}_n$
on the homology of each space $K_i$  in 
$\calK$.

\begin{definition} \label{Delta_nPersistentGroups}
For every $0 \leq i \leq j \leq N$,
$p\ge 0$ and $n\geq 1$, the \textbf{$\bm{p^{\text{\rm \textbf{th}}}}$ $\bm{\De_n}$-persistent group} between $K_i$ and $K_j$ is the vector space
\begingroup\makeatletter\def\f@size{14}\check@mathfonts
$$ (\De_n)_p^{i,j}\left( \mathcal K \right) \coloneqq \Image {f^{i,j}_p}_{|\cap_{k=i}^j \Ker \left( \De_n^k \circ f^{i,k}_p \right)}.$$
\endgroup
These are the ${\bm A_\infty}$\textbf{-persistence groups} in terms of 
transferred $A_\infty$-coalgebras in homology.
\end{definition}

In $A_\infty$-persistence \cite{Belchi-Murillo15},
we study the evolution of vector subspaces such as 
$V_i \coloneqq \Ker {\Delta_n}_{| H_p(K_i)}$ $\subseteq H_p(K_i)$.
The problem is that the maps 
$
f^{i,i+1}_p\colon H_p(K_i)
\longrightarrow H_p(K_{i+1})
$
do not restrict, in general, to maps
$
\xymatrix{
V_i \ar[r] & V_{i+1}
}$
(see for instance 
\cite[Thm. 3.1]{Belchi-Murillo15}),
therefore \emph{not} producing a persistence module
$
\xymatrix{
V_0 
\ar[r] & V_1
\ar[r] & \ldots
\ar[r] & V_N.
}$
Despite that, we can still completely understand the  
persistent groups in Def. \ref{Delta_nPersistentGroups} by using zigzag techniques \cite{Carlsson-de_Silva10}. Indeed, here is the $A_\infty$ counterpart 
of Thm. \ref{Thm:Fundamental_thm_PH}:

\begin{theorem}\label{COR:Delta_n-bars}
\cite[Cor. 2.9, Def. 2.8]{Belchi-Murillo15},
For any pair of integers $p \geq 0$ and $n \geq 1$, 
there exists a unique
multiset $M_{p, n}$ of
intervals of the form $[i,j]$ for $0 \leq i < j \leq N$,
such that	for all \,$0 \leq i \leq j \leq N$,
the dimension
$\dim (\De_n)_p^{i,j}\left( \mathcal K \right)$
equals the number of intervals
in $M_{p, n}$ 
(counted with multiplicities)
which contain the interval $[i,j]$.

In particular, 
$\dim \Ker {\De^i_n}_{|H_p(K_i)}$
equals the number of intervals
in $M_{p, n}$ 
(counted with multiplicities)
which contain the integer $i$.
\end{theorem}

%\begin{definition}
%\label{Def:Ainfty_Barcode}
The multiset $M_{p, n}$ in Thm.
\ref{COR:Delta_n-bars}
is called the 
\textbf{$\bm{p^{\text\rm \textbf{th}}}$ 
$\bm{\Delta_n}$-barcode} of $\calK$.
%\end{definition}
The construction of this barcode involves zigzag decompositions and in \S \ref{compatibleAinftyStructures} 
we will deal with a hazard this presents.

%%%%%%%%%%%%%%%%%%%%%%%%%%%%%%%%%%%%%
%%%%%%%%%%%%%%%%%%%%%%%%%%%%%%%%%%%%%
%%%%%%%%%%%%%%%%%%%%%%%%%%%%%%%%%%%%%
\section{Topological meaning of $A_\infty$-persistence groups}
\label{sec_topologicalMeaning}

It is well known that a torus, a Klein bottle and something as simple as a wedge of spheres $\bS^1 \vee \bS^2 \vee \bS^1$ have isomorphic homology groups with $\bZ_2$ coefficients, but that the cup product tells these 3 spaces apart.
These objects are not only mathematical amusements -- they can actually model datasets. For instance, a quasi-periodic signal can produce a sliding window embedding which is dense in a torus \cite{Perea16_toroidal} and high-contrast patches in grey-scale images can produce a point cloud which is dense in a space whose 2-skeleton is a Klein bottle \cite{Carlsson-Ishkhanov-de_Silva-Zomorodian08}.
Since persistence can be used to study these datasets (as recalled in Fig. \ref{fig:DiagramaResumen}), it is therefore relevant to have a persistence approach to cup product.
A particular case of $A_\infty$-persistence yields one such approach, since the multiplication $\mu_2$ in a transferred $A_\infty$-algebra
is precisely the cup product. Inasmuch as this product is well understood, in this section we will focus on understanding what $A_\infty$-persistence can achieve with the higher operations $\Delta_n$ and $\mu_n$, for $n>2$.

The key to understanding the persistent homology groups $H^{i,j}_p(\calK)$ is to know what $\dim H^{i,i}_p(\calK)$ means. Similarly, the key to understanding $A_\infty$-persistence groups as in Definition \ref{Delta_nPersistentGroups}, $(\Delta_n)^{i,j}_p(\calK)$, is to know what $\dim (\Delta_n)^{i,i}_p(\calK)$ means.
The meaning of the former is easy to grasp, since $\dim H^{i,i}_p(\calK)$ reduces to the Betti number $\beta_p(K_i)$. The meaning of the later is a little more involved, so we now explore some of the topological information encoded in the dimension of $(\Delta_n)^{i,i}_p(\calK) = \Ker {\Delta_n}_{|H_p(K_i)}$ (and dually the dimension of the cokernel of the map ${{\mu_n}_{|H^{p_1}(K_i) \otimes \ldots \otimes
H^{p_n}(K_i)}}$). We will start by showing that, in the worst case, different $A_\infty$-structures transferred by the same space can lead to different values of 
$\dim \Ker {\Delta_n}_{|H_p(K_i)}$ if $n>2$
(Ex. \ref{prop_De_n-not-invariant}).
However, we identify situations in which this does cannot happen. In such situations, these values provide information on the homotopy type of topological spaces, 
beyond that of the standard cohomology algebra
(Thm. \ref{THM:min-invariant}, 
Cor. \ref{COR:min-invariant} and Prop. \ref{Prop:Borromean-m3} and \ref{Prop:ExtendingBorromean}).

Some results in this section involve folklore notions
of either rational homotopy theory or
homological perturbation theory, but to the best of our knowledge, 
there is no proof in the literature of any of the results we will prove here.

Let us start by stating that numbers like
$\dim \Ker {\Delta_n}_{|H_p(X)}$ can give different values depending on the chosen transferred $A_\infty$-structure.

\begin{example}\label{prop_De_n-not-invariant}
Let $X$ be wedge of the complex projective plane and a 7-sphere
$$
X=\mathbb{C} P^2 \vee \mathbb{S}^7.
$$
Then, the
reduced rational homology of $X$, 
$\widetilde H_*(X;\bQ)$,
admits two (isomorphic)
transferred $A_\infty$-coalgebra structures
$\{\Delta_n\}_n$ and $\{\Delta'_n\}_n$
such that
$$\dim \Ker {\Delta_3}
\neq 
\dim \Ker {\Delta'_3}. $$
\end{example}

This example (which we explain in detail in \S \ref{sec_appendix})
could lead some to think that 
there is no use in using
quantities such as $\dim \Ker {\Delta_n}_{|H_p(X)}$
to distinguish topological spaces, but there is no need to 
\emph{throw the baby out with the bath water.}
The rest of this section is devoted to prove that in some
circumstances these numbers can still tell apart
non-homotopically-equivalent topological spaces whose
homology groups and even cohomology algebras
are isomorphic.

We continue by stating in Thm. \ref{THM:min-invariant}
(which we prove in \S \ref{sec_appendix})
that isomorphic $A_\infty$-coalgebras induce the same numbers
$\dim \Ker \De_m$ for several values of $m$.

\begin{theorem}
\label{THM:min-invariant}
Let $ \{\De_n\}_n $ and 
$ \{\De'_n\}_n $
be two isomorphic minimal $A_\infty$-coalgebra structures
on a graded vector space $C$ such that $C_p = 0$ for all $p<0$.
Let us set
$$
k \coloneqq
\begin{cases}
 \min \{ n \,|\, \Delta_n \neq 0 \} , \qquad \text{if } \{ n \,|\, \Delta_n \neq 0 \} \neq \emptyset \\
+\infty, \qquad \text{otherwise}
\end{cases}
$$
and
$$
k' \coloneqq
\begin{cases}
 \min \{ n \,|\, \Delta'_n \neq 0 \} , \qquad \text{if } \{ n \,|\, \Delta'_n \neq 0 \} \neq \emptyset \\
+\infty, \qquad \text{otherwise.}
\end{cases}
$$
Then $k=k'$ and
$$ \dim \Ker {\De_m}_{|C_p} = 
\dim \Ker {\De'_m}_{|C_p}, $$
for all integers $m\leq k$ and $p\geq 0$.
\end{theorem}

Applying Thm. \ref{THM:min-invariant} to transferred $A_\infty$-structures on $C=\widetilde H_*(X)$ yields sufficient conditions for the numbers  $ \dim \Ker {\De_m}_{|\widetilde H_p(X)} $ to be homotopy invariants. We state this as the following corollary:

\begin{corollary}
\label{COR:min-invariant}
Let $ \{\De_n\}_n $ be a transferred 
$A_\infty$-coalgebra structure on the reduced homology of a space $X$, and let us set
$$
k \coloneqq
\begin{cases}
 \min \{ n \,|\, \Delta_n \neq 0 \} , \qquad \text{if } \{ n \,|\, \Delta_n \neq 0 \} \neq \emptyset \\
+\infty, \qquad \text{otherwise.}
\end{cases}
$$
Then, $k$ and the numbers
$ \dim \Ker {\De_m}_{|\widetilde H_p(X)} $ (for integers $m \leq k$ and $p\geq 0$) are
independent of the choice of transferred $A_\infty$-coalgebra structure on $\widetilde H_*(X)$. 
Moreover, since homotopy equivalent spaces induce isomorphic transferred $A_\infty$-coalgebras, $k$ and every such $ \dim \Ker {\De_m}_{|\widetilde H_p(X)} $ are invariants of the homotopy type of $X$.
\end{corollary}

Using Thm. \ref{THM:min-invariant}, one can infer similar statements to 
Cor. \ref{COR:min-invariant}
for transferred $A_\infty$-coalgebras
in (non-reduced) homology and dualize them to
transferred $A_\infty$-algebras in cohomology (reduced or not).
Cor. \ref{COR:min-invariant} shows $A_\infty$-operations that vanish up to a certain $n\geq1$ so that the possible
non-vanishing at level $n+1$ becomes important.
Following this idea, we now show how the numbers
$$\dim \Coker {\mu_n}_{|H^{p_1}(X) 
\otimes \ldots \otimes
H^{p_n}(X)}$$
can sometimes find extra topological information. 
We do this by using the connection between $A_\infty$-structures,
Massey products and higher linking numbers.
We will consider links in $\mathbb{S}^3$ 
(\emph{i.e.}, collections of embeddings of a circumference
$\mathbb{S}^1$
into $\mathbb{S}^3$ which do not intersect)
and embeddings of arbitrary spheres in a sphere
of a higher dimension.
Let us start with 2-component links:
the unlinked $\mathcal{U}_2$
of Fig. \ref{fig:unlinked2Rings} 
and the  Hopf link $\calH$ 
of Fig. \ref{fig:HopfLink}.    
    \begin{figure}[h!]
    \centering
    \begin{subfigure}[b]{0.2\textwidth}
        \includegraphics[width=\textwidth]{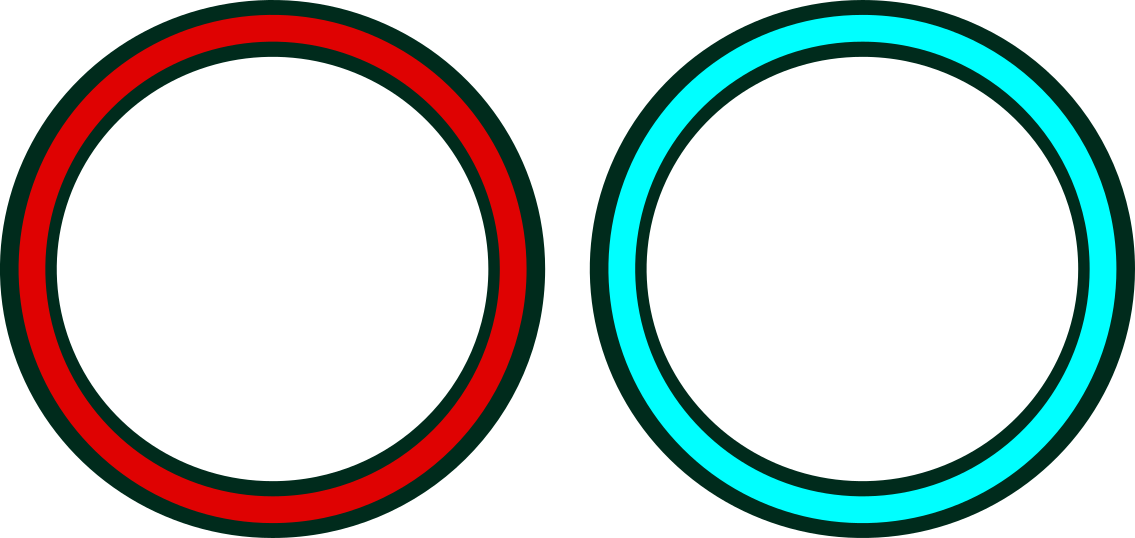}
        \caption{}
        \label{fig:unlinked2Rings}        
    \end{subfigure}
    \hspace{4cm}
    ~ %add desired spacing between images, e. g. ~, \quad, \qquad, \hfill etc. 
    \begin{subfigure}[b]{0.14\textwidth}
        \includegraphics[width=\textwidth]{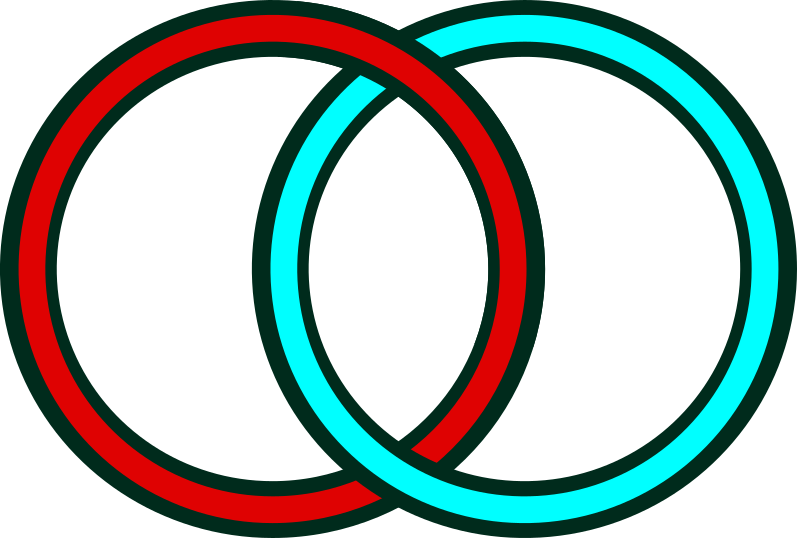}
        \caption{}
        \label{fig:HopfLink}        
    \end{subfigure}

     \caption{Betti numbers cannot distinguish the two unlinked annuli $\mathcal{U}_2$ of (a), from the Hopf link $\calH$ (two disjoint annuli linked as in (b)). However, the standard cohomology algebra on 
$H^*(\bS^3-\mathcal{U}_2)$ and $H^*(\bS^3-\calH)$ 
can tell them apart. In particular, since all the information in the stardard cohomology algebra is contained in any transferred $A_\infty$-algebra, we conclude that we can distinguish these two links by using $A_\infty$-algebras.}
    \label{fig:UnlinkedVsHopf}
\end{figure}
It is easy to see that
$ H_p(\mathcal{U}_2) \cong H_p(\calH) $
for all $p\geq 0$ and, using 
Alexander Duality,
$ H_p(\bS^3-\mathcal{U}_2) \cong H_p(\bS^3-\calH) $
holds as well.
Recall that the cup product in the complement in $\bS^3$ of two knots can detect their linking number
\cite[\S 5.D]{Rolfsen76}. In this case,  
the fact that 
$\mathcal{U}_2$ has a 0 linking number but 
$\calH$ does not, implies that
$$ 0 = \smile: H^1(\bS^3-\mathcal{U}_2) \otimes 
H^1(\bS^3-\mathcal{U}_2) \longrightarrow
H^2(\bS^3-\mathcal{U}_2),$$
whereas
$$ 0 \neq \smile: H^1(\bS^3-\calH) \otimes 
H^1(\bS^3-\calH) \longrightarrow
H^2(\bS^3-\calH).$$
Hence, the cup product, which we recall is the operation
$\mu_2$ in any 
transferred 
$A_\infty$-algebra $\{\mu_n\}_{n \geq 1}$,
can tell apart these two links.

Let us look now at 
two 3-component links:
the unlinked $\mathcal{U}_3$
of Fig. \ref{fig:Unlinked_Ink}
and the  Borromean rings $\mathcal{B}$
of Fig. \ref{fig:Borromean_Ink}. The latter satisfy that if we remove any of
the rings, the other two become unlinked, but the three of them considered at once cannot be pulled apart.
It is easy to see that
$H_p(\mathcal{U}_3) \cong H_p(\mathcal{B}) $
and
$ H_p(\bS^3-\mathcal{U}_3) \cong H_p(\bS^3-\mathcal{B}) $
holds for all $p\geq 0$.
Moreover, for both $\mathcal{U}_3$ and $\mathcal{B}$,
if we remove one of the 
circumferences, 
the other two become unlinked.
Hence, since the remaining two circumferences
are unlinked, 
their linking number is 0 and 
the cup product of
the Alexander duals 
of those fundamental classes vanishes.
The fact that this pairwise unlinkedness holds for the three circumferences
implies that the cup product
$$
\smile: H^1(X) \otimes 
H^1(X) \longrightarrow
H^2(X)
$$
is 0 for both 
$X = \bS^3-\mathcal{U}_3$ and
$X = \bS^3-\mathcal{B}$. Indeed, we
need a 3-fold operation such as the triple Massey product
to detect the 3-fold linkedness of the
Borromean rings.

The triple Massey product \cite{Uehara-Massey57} on the cohomology
algebra of a space, 
$(H^*(X), \smile)$,
is defined only for those triples of classes
$\alpha \in H^p(X),
\beta \in H^q(X),
\gamma \in H^r(X)$ such that
$$ \alpha \smile \beta = 0
\hspace{10mm}
\text{and}
\hspace{10mm}
\beta \smile \gamma = 0,$$
for which such triple Massey product 
$\left\langle \alpha, \beta, \gamma \right\rangle$
is a specific \emph{subset}
$$\left\langle \alpha, \beta, \gamma \right\rangle \subseteq H^{p+q+r-1}(X).$$

\begin{figure}[h]
    \centering
    \begin{subfigure}[b]{0.2\textwidth}
        \includegraphics[width=\textwidth]{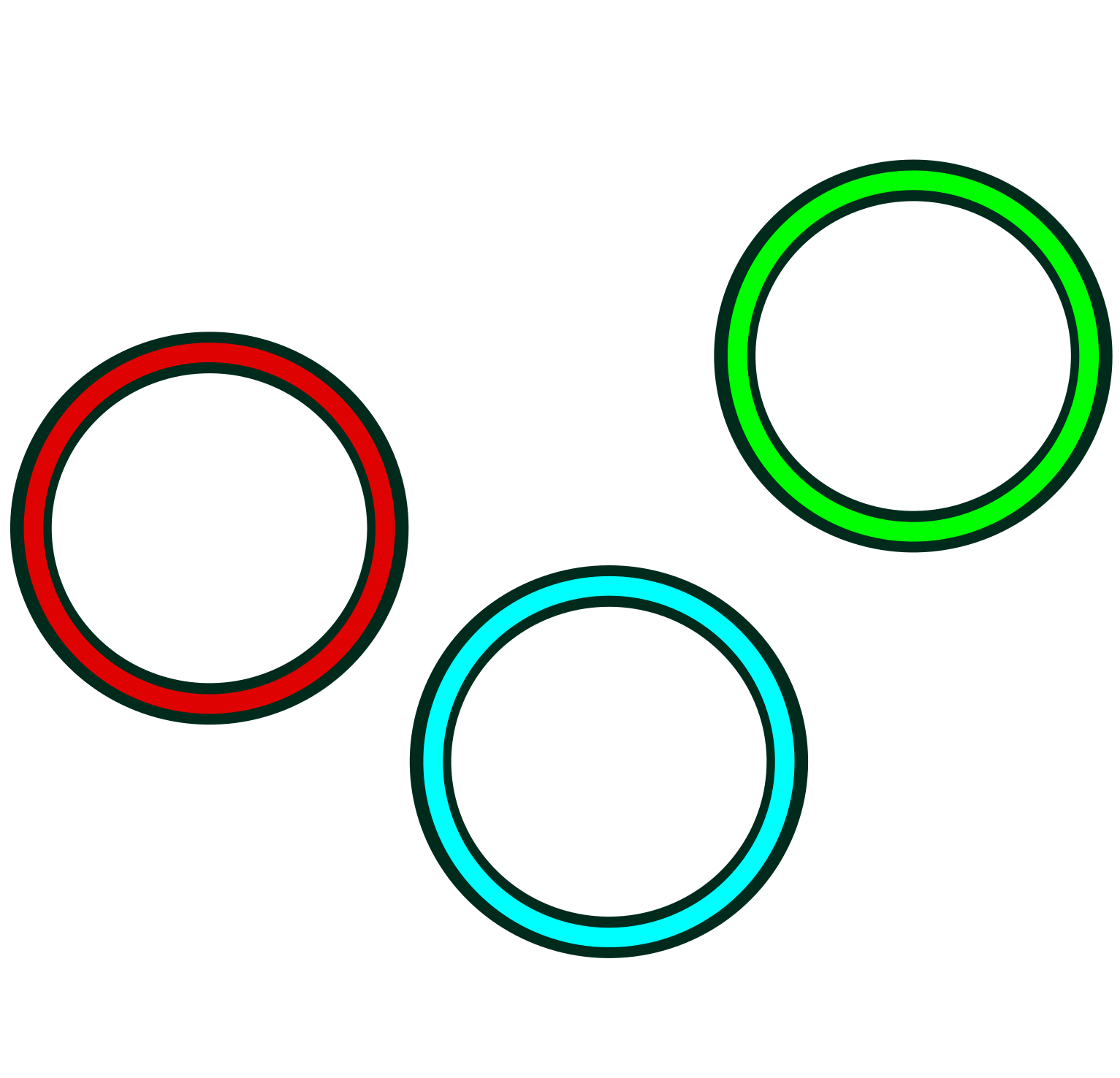}
        \caption{}
        \label{fig:Unlinked_Ink}        
    \end{subfigure}
    \hspace{4cm}
    ~ %add desired spacing between images, e. g. ~, \quad, \qquad, \hfill etc. 
    \begin{subfigure}[b]{0.2\textwidth}
        \includegraphics[width=\textwidth]{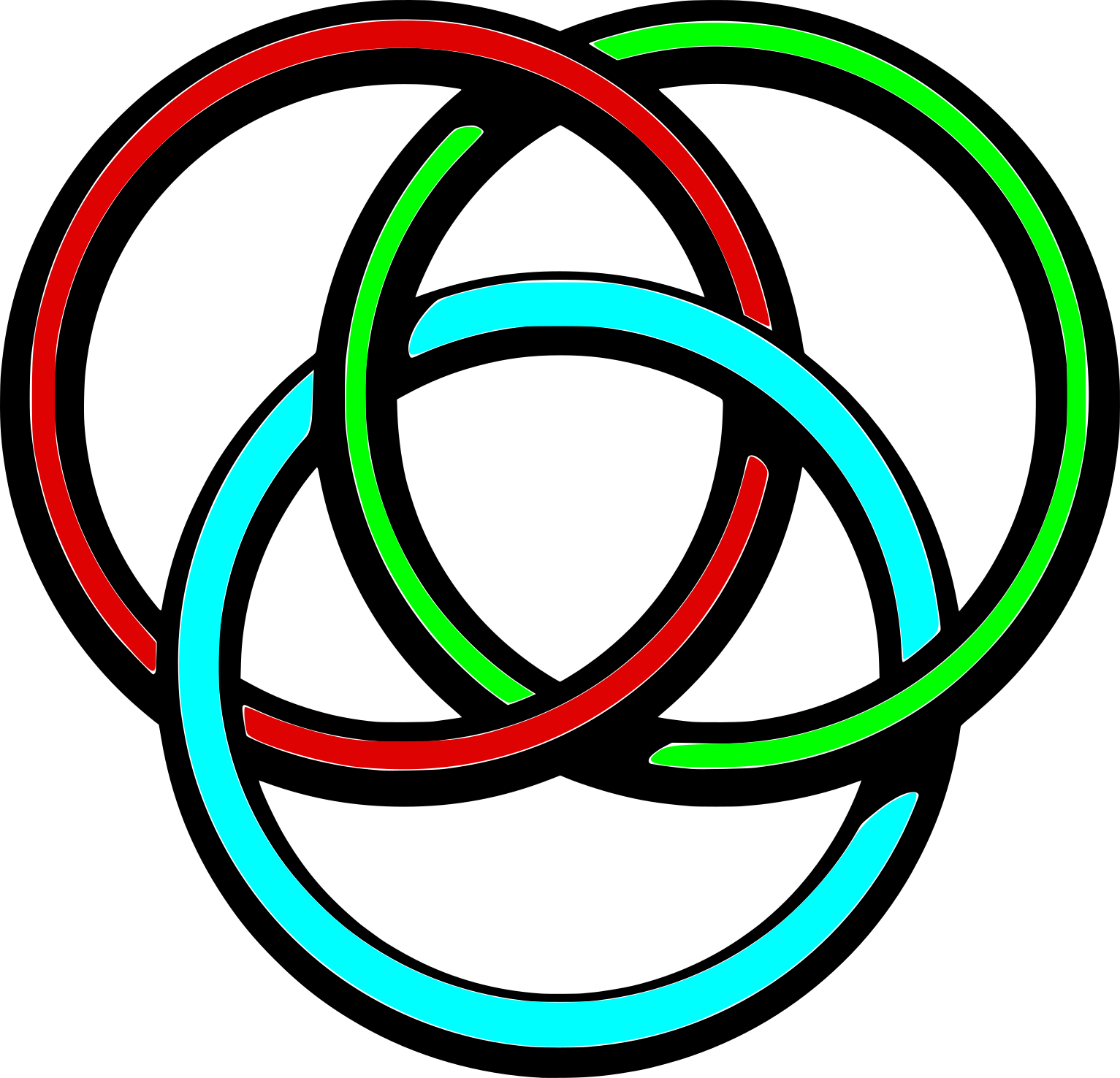}
        \caption{}
        \label{fig:Borromean_Ink}        
    \end{subfigure}

    \caption{Betti numbers and the cup product cannot distinguish the three unlinked annuli $\mathcal{U}_3$ of (a), from the Borromean rings $\mathcal{B}$ (three disjoint annuli linked as in (b)). However, using $A_\infty$-algebras on 
$H^*(\bS^3-\mathcal{U}_3)$ and $H^*(\bS^3-\mathcal{B})$ can tell them apart (see
Prop. \ref{Prop:Borromean-m3}).}
    \label{fig:UnlinkedVsBorromean}
\end{figure}

The second ingredient we need to prove
Thm. \ref{Prop:Borromean-m3} is an interesting relation between
$A_\infty$-structures and Massey products that tells us that,
given a topological space $X$,
the higher $A_\infty$-multiplication
$\mu_3$ on any transferred $A_\infty$-algebra $\{\mu_n\}_n$ on $H^*(X)$ 
computes Massey products up to a sign:

\begin{theorem}
\cite[\S 12]{Stasheff70}, \cite[Thm 3.1, Cor A.5]{Lu-Palmieri-Wu-Zhang09}.
\label{m_n-compute-Massey}
For any transferred $A_\infty$-algebra
$\{\mu_n\}_n$
on the cohomology of a space $X$,
if $\alpha, \beta, \gamma \in H^*(X)$
are cohomology classes for which
the triple Massey product
$ \left\langle \alpha, \beta, \gamma \right\rangle $
is defined,
then
$$ (-1)^b \mu_3 \left( \alpha \otimes \beta
\otimes \gamma \right)
\in \left\langle \alpha, \beta, \gamma \right\rangle,$$
where 
$b = 1+ |\beta|.$
\end{theorem}

Despite any two transferred $A_\infty$-algebras on $H^*(X)$
being isomorphic, their higher multiplications
$\mu_3$ may behave slightly differently
(as seen in the dual scenario in Ex. \ref{prop_De_n-not-invariant}),
but this result gives us ways to take 
advantage of some particular choices of $\mu_3$
as in the following proposition.

\begin{proposition}
\label{Prop:Borromean-m3}
Let \,$\mathcal{U}_3$ and $\mathcal{B}$ be as above, 
the 3-component trivial unlinked link 
and the Borromean rings, respectively.
Then,  
$H^* \coloneqq H^*(\bS^3-\mathcal{U}_3)$
admits a
transferred $A_\infty$-algebra structure
$\{\mu_n\}_n$ 
such that
$$ {\mu_3}_{|H^1 \otimes H^1 \otimes H^1} = 0,$$
whereas $H^* \coloneqq H^*(\bS^3-\mathcal{B}_3)$ does not.
\end{proposition}

\begin{proof}
Let us start with the unlinked $\mathcal{U}_3$.
Using stereographic projection, it is easy to see that
$\bS^3 - \mathcal{U}_3$ is homotopy equivalent
to $\bR^3$ minus 
two circumferences and an infinitely long straight line, 
where these 3 objects sit in 3 disjoint regions of $\bR^3$.
The space we just described is, in turn, homotopy equivalent to 
the wedge of spheres 
$$ X = \bS^2 \vee \bS^2 \vee \bS^1 \vee 
\bS^1 \vee \bS^1.$$
As any wedge of spheres,
$X$ admits a
CW decomposition that yields
a cellular cochain complex with trivial coboundary operator
$(C^{*}(X), \delta = 0)$.
Thus,
$ C^{p}(X) \cong H^{p}(X) $
for all $p\geq 0$ and
we can therefore choose
$\pi$ and $\iota$ to be isomorphisms 
that make $\phi$ to be the zero map
in a diagram dual to (\ref{ecuacion1}),
	$$
	\xymatrix{ \ar@(ul,dl)@<-5.5ex>[]_\phi & C^{*}(X) \ar@<0.75ex>[r]^-\pi & H^{*}(X). \ar@<0.75ex>[l]^-\iota }
	$$
Applying Remark \ref{Rmk_HTT} in terms of $A_\infty$-algebras produces
a transferred $A_\infty$-algebra 
$\{\mu_n\}_n$ such that
$\mu_n = 0$ for all $n\neq 2$.
Therefore, $H^*(\bS^3-\mathcal{U}_3)$
admits a
transferred $A_\infty$-algebra
$\{\mu_n\}_n$ 
such that
$ \mu_3 = 0.$

Let us look now at the Borromean rings
$\mathcal{B}$.
On the one hand,
W. S. Massey proved in \cite{Massey68/98} 
that the Alexander dual of the fundamental classes
of the Borromean rings, 
which form a basis $\{\alpha, \beta, \gamma \}$ of 
$H^1(\bS^3-\mathcal{B})$,
satisfy 
$$ 0 \notin \left\langle \al, \be, \gamma \right\rangle. $$
On the other hand, 
Thm. \ref{m_n-compute-Massey}
tells us that
for any 
transferred $A_\infty$-algebra 
$\{\mu_n\}_n$ on $H^*(\bS^3 - \mathcal{B})$,
the following holds:
$$ \mu_3 \left( \al \otimes \be
\otimes \gamma \right)
\in \left\langle \al, \be, \gamma \right\rangle.$$
Both things together yield the conclusion
$ \mu_3 \left( \al \otimes \be
\otimes \gamma \right) \neq 0.$
\end{proof}

Let us end this section by proving one last result,
extending Prop.
\ref{Prop:Borromean-m3}, 
and mentioning that other examples
using
$ {\mu_n}_
{| H^{p_1}(X) \otimes \ldots \otimes H^{p_n}(X)}
$
for other values of $n, p_1, \ldots, p_n$
can be analogously obtained.

Let 
$m=p+q+r$,
$x=(x_1, \ldots, x_p),
y=(y_1, \ldots, y_q),$
and $z=(z_1, \ldots, z_r)$.
Consider three spaces $S_1, S_2, S_3$ 
in 
$\bR^m$
(homeomorphic to three spheres of different dimensions) determined by the
pairs of equations
\begin{align*}
\text{\rm{$(q+r-1)$-dimensional sphere $S_1\colon$}} & &
x=0, & &
\left\| y \right\|^2 + 
\frac{\left\| z \right\|^2}{4} =1 \\
\text{\rm{$(p+r-1)$-dimensional sphere $S_2\colon$}} & &
y=0, & &
\left\| z \right\|^2 + 
\frac{\left\| x \right\|^2}{4} =1 \\
\text{\rm{$(p+q-1)$-dimensional sphere $S_3\colon$}} & & 
z=0, & &
\left\| x \right\|^2 + 
\frac{\left\| y \right\|^2}{4} =1
\end{align*}
Notice that the case $p=q=r=1$ yields Borromean
rings.

Recall that the Jordan-Brower separation theorem
tells us that for any space $X$ homeomorphic to 
$\bS^{m-1}$, the complement
$\bR^m-X$ consists of two path components --
the \emph{inside} and the \emph{outside}
of $X$. Any two of the spheres $S_i$
are separated by a space homeomorphic
to $\bS^{m-1}$, which plays an analogous role to 
pairwise unlinkedness for knots. For instance,
$S_1$ lies outside the sphere
$$
\frac{\left\| x \right\|^2}{3^2} + 
\frac{\left\| y \right\|^2}{(1/2)^2} + 
\frac{\left\| z \right\|^2}{(3/2)^2} = 1,
$$
while $S_3$ lies inside it.
Now we can use Alexander duality
in the exact same way as before.
Letting $K$ be the disjoint union of
$S_1, S_2$ and $S_3$,
consider the cohomology classes
${\alpha \in \widetilde H^p(\bS^m - K)
\cong} \widetilde H_{q+r-1}(K),
{\beta \in \widetilde H^q(\bS^m - K)
\cong} \widetilde H_{p+r-1}(K),$ and
${\gamma \in \widetilde H^r(\bS^m - K)
\cong}$\\
\noindent $ \widetilde H_{p+q-1}(K),$
corresponding to the dual of the fundamental classes of $S_1,S_2$ and $S_3$.
The fact that these spheres can be separated
as explained makes the pairwise
cup product of these cohomology classes
vanish, so that their triple Massey product is defined.
Furthermore, W. S. Massey also proved in
\cite[\S 4]{Massey68/98} that
this Massey product cannot contain the 0 cohomology class.
Hence, the following result follows from 
using the proof of Prop. \ref{Prop:Borromean-m3}
\emph{mutatis mutandis}:

\begin{proposition}
\label{Prop:ExtendingBorromean}
With the notation above, 
any transferred $A_\infty$-algebra
$\{\mu_n\}_n$ 
on $H^*(\bS^m - K)$
will have
$$ {\mu_3}_{|H^p \otimes H^q \otimes H^r} \neq 0,$$
where $H^k$ denotes $H^k(\bS^m - K)$.
\end{proposition}

As we explained, the triple Massey product is defined whenever certain
2-fold products vanish (namely, when the cup products are 0).
Similarly,
for every $n\geq 4$ there is also an $n$-fold
Massey product \cite{Massey58} defined
whenever the $(n-1)$-fold Massey products vanish
(in the sense that those sets contain the 0 cohomology class).
On the other hand, the operation $\mu_n$ of a transferred $A_\infty$-algebra computes $n$-fold Massey product similarly to the case $n=3$
shown in Thm. \ref{m_n-compute-Massey}
(see \cite[Thm 3.1, Cor A.5]{Lu-Palmieri-Wu-Zhang09} and more recently \cite{Pepe17}).
Hence,
one can think of results similar to
Prop. \ref{Prop:Borromean-m3}
and \ref{Prop:ExtendingBorromean}
for $n$ bodies instead of just 3, 
which will involve information about 
the $A_\infty$-operation $\mu_n$.
There is also a 
relation between $L_\infty$-algebras 
and higher order Whitehead products of a similar
flavour to that of $A_\infty$-algebras
and higher order Massey products.
The interested reader can explore this in	
\cite[Thm. V.7(7)]{Tanre83} and more recently in
\cite{Belchi-Buijs-Moreno-Murillo17}.

Finally, notice that results such as Prop. \ref{Prop:Borromean-m3}
and \ref{Prop:ExtendingBorromean}
can be exploited by $A_\infty$-persistence
to study datasets with an underlying (low or high-dimensional) link structure. 
%Other approaches on combining persistence with knots and links can be found in  
%\cite{Brendel-Dlotko-EllisEtAl15, Edelsbrunner-Zomorodian03}.
% I can remove this last 2 citations if the reference list is too long.
Also, if we know in which particular context we are working on, we can use the results in this section to choose the right $n$ to focus on when using $A_\infty$-persistence.

%%%%%%%%%%%%%%%%%%%%%%%%%%%%%%%%%%%%%
%%%%%%%%%%%%%%%%%%%%%%%%%%%%%%%%%%%%%
%%%%%%%%%%%%%%%%%%%%%%%%%%%%%%%%%%%%%
\section{Compatible $A_\infty$-structures to avoid zigzag ambiguities}
\label{compatibleAinftyStructures}

Given the sequence of topological spaces and continuous maps
$$
\xymatrix{
\calK\colon  & 
K_0
\ar[r] & K_1
\ar[r] & \ldots
\ar[r] & K_{N},
}$$
let 
$\{\De_m^i\}_{m\geq1}$ 
denote a transferred $A_\infty$-coalgebra on $H_*(K_i)$,
for each $0 \leq i \leq N$,
and let $n \geq 1$ and $p \geq 0$
be two integers. Once all this is fixed, to simplify notation, write
$\Delta^i$ to represent the restriction of ${\Delta_n^i}$ to ${H_p(K_i)}$.
Let us look at a particular example for a moment. 

\begin{example}
\label{Ex_2alignedBars}
Let us assume
that $N=999$, 
$\Ker \Delta^{500} = \{0\}$
and that there exists a homology class
$\alpha \in H_p(K_0)$ such that
$$
0 \neq f^{0,j}\alpha \in \Ker \Delta^j, \text{ for all } j \neq 500, 
\hspace{10mm} \text{and} \hspace{10mm} 
0 \neq f^{0,500}\alpha \notin \Ker \Delta^{500}.
$$
Notice that a homology class can indeed behave like this,
as seen in 
\cite[\S3]{Belchi-Murillo15}. 
It is straightforward to check that 
the existence of the class $\alpha$,
whose image by $f^{0,j}$ falls into $\Ker \Delta^j - \{0\}$ at $99.9\%$ of the terms
$K_j$ forming $\mathcal{K}$, results in the creation,
in the $p^\text{th}$ $\Delta_n$-barcode of $\calK$, of two (apparently unrelated) intervals $I=[0,499]$ and $J=[501,999]$,
% (see Fig. \ref{fig_bars0to499_501to999}), 
each containing at most 500 integers, which amounts to only $50\%$ of the terms in $\mathcal{K}$.
%\caption{A class $\Delta_n$-awake
% at $99.9\%$ of the terms that form $\mathcal{K}$
% results in two (apparently unrelated) intervals in this barcode, each containing 
% only at most $50\%$ of the number of terms that form $\mathcal{K}}$.
\end{example}

Ex. \ref{Ex_2alignedBars} shows that 
the $\Delta_n$-persistent groups and their corresponding barcodes
are not as faithful as we would like in describing
the evolution of subspaces such as $\Ker \De_n$
along
$\calK$.
This is due to the need of using
zigzag persistence \cite{Carlsson-de_Silva10} in the construction of the $\Delta_n$-barcodes.
One way to avoid this issue is to 
choose sequences $\calK$ and 
transferred $A_\infty$-structures along $\calK$ in a coherent 
manner to allow us to compute $A_\infty$-persistence via persistence modules
rather than via more general zigzag modules.
In terms of the $p^\text{th}$ $\Delta_n$-barcode
in homology, 
this boils down to being able to compute
a transferred $A_\infty$-coalgebra
$\{\Delta^i_m\}_m$ on $H_*(K_i)$, for each $K_i$ in $\calK$,
such that the following will hold for every $0 \leq i < N \colon$
\begin{align}
\label{eq_assumptionfKerInKer}
f_p^{i,i+1} \left( \Ker \De^i \right)
\subseteq \Ker \De^{i+1}.
\end{align}
In this section we show a way to construct filtrations $\calK$
and the corresponding transferred $A_\infty$-coalgebras
so that assumption (\ref{eq_assumptionfKerInKer}) holds
(see Thm. \ref{THM:Coherent-choices-Quillen}), and prove the
barcode decomposition theorem of $A_\infty$-persistence
for the case of such compatible choices of $A_\infty$-structures
(see Thm. \ref{Thm_Fund_Thm_A_infty-p_Functorial})
to stress how much things simplify, even at a theoretical level, by bypassing 
the need for zigzag machinery.

Here is a construction exhibiting an example
of a choice of $A_\infty$-structures that make assumption (\ref{eq_assumptionfKerInKer}) hold.

\begin{theorem}
\label{THM:Coherent-choices-Quillen}
Let $K_0$ be a 1-connected
CW complex and let $N > 0$ be an integer. 
For
all $0<i\leq N$,
let $K_i$ denote the wedge $K_{i-1} \vee \mathbb{S}^{n_i}$, for some $n_i > 1$.
Then there exists
a transferred $A_\infty$-coalgebra
$\left( \widetilde H_*(K_i; \bQ), \{ \De_n^i \}_n
\right)$
on the reduced rational homology of $K_i$, for each $0\leq i \leq N$,
 such that the following inclusion holds for every $n\geq 1$,
 $p\geq 0$ and $0 \leq i < N$:
$$f_p^{i,i+1} \left( \Ker \De^i \right)
\subseteq \Ker \De^{i+1},$$
where $\Delta^i$ denotes the restriction of ${\Delta_n^i}$ to 
${\widetilde H_p(K_i; \bQ)}$
and $f_p^{i,i+1}: \widetilde H_p(K_i; \bQ) \longrightarrow \widetilde H_p(K_{i+1}; \bQ)$ denotes the map induced by the inclusion $K_i \xhookrightarrow{} K_{i+1}$.
\end{theorem}

We save the proof of Thm. \ref{THM:Coherent-choices-Quillen}
for \S \ref{sec_appendix} and
proceed to prove the following particular
case of the fundamental decomposition theorem of $A_\infty$-persistence
\cite[Thm. 2.7]{Belchi-Murillo15}, to show how much the assumption
(\ref{eq_assumptionfKerInKer})
simplifies things; compare this to the proof 
in \cite[Thm. 2.7]{Belchi-Murillo15}.

\begin{theorem}
\label{Thm_Fund_Thm_A_infty-p_Functorial}
Fix some integers $p\geq0$ and $n\geq 1$. 
If $\Delta^i$ denotes the restriction of ${\Delta_n^i}$ to ${H_p(K_i)}$
and $f_p^{i,i+1} \left( \Ker \De^i \right)
\subseteq \Ker \De^{i+1}$ holds for all
$0 \leq i < N$, then there exists a unique
multiset $M_{p, n}$ of
intervals of the form $[i,j)$ for $0 \leq i < j \leq N$,
and intervals of the form $[i,\infty)$ for $0 \leq i \leq N$,
such that	for all \,$0 \leq i \leq j \leq N$,
the dimension
$\dim (\De_n)_p^{i,j}\left( \mathcal K \right)$
equals the number of intervals
in $M_{p, n}$ 
(counted with multiplicities)
which contain the interval $[i,j]$.

In particular, 
$\dim \Ker {\De^i_n}_{|H_p(K_i)}$
equals the number of intervals
in $M_{p, n}$ 
(counted with multiplicities)
which contain the integer $i$.
\end{theorem}

\begin{proof}
Fix a transferred $A_\infty$-coalgebra structure
$\{\De_n^i\}_n$ on $H_*(K_i)$,
for all $0 \leq i \leq N,$
and fix two integers $n\geq 1$
and $p \geq 0$.
If the inclusion
$f^{i,i+1}_p \left( \Ker \De^i \right)
\subseteq \Ker \De^{i+1}$
holds for all
$0 \leq i < N$, then 
the maps $f^{i,j}_p$
in the persistence module
$$ 
\xymatrix{
H_p(K_0)
\ar[r]^-{f^{0,1}_p} & H_p(K_1)
\ar[r]^-{f^{1,2}_p} & \ldots
\ar[r]^-{f^{N-1,N}_p} & H_p(K_N)
}$$
restrict to maps $g^{i,j}_p$ that form the persistence (sub)module
\begin{equation}
\label{persistence_submodule}
\xymatrix{
\Ker \De^0
\ar[r]^-{g^{0,1}_p} & \Ker \De^1
\ar[r]^-{g^{1,2}_p} & \ldots
\ar[r]^-{g^{N-1,N}_p} & \Ker \De^1.
}
\end{equation}

\noindent Since $(\De_n)_p^{i,j}\left( \mathcal K \right) = \Image g^{i,j}_p$,
applying Lemma \ref{lemma_persModsSplit}
to (\ref{persistence_submodule}) finishes the proof.
\end{proof}

\begin{remark}
In the terminology introduced in 
\cite[Def. 2.5]{Belchi-Murillo15}, 
Ex. \ref{Ex_2alignedBars} shows that a class that $\De_n$-falls asleep and
$\De_n$-wakes up again would be represented in the $\De_n$-barcode as apparently independent intervals, each corresponding to a different period in which the class has been $\De_n$-awake, and assumption (\ref{eq_assumptionfKerInKer})
would amount to making sure no class can $\De_n$-wake up once it has 
$\De_n$-fallen asleep.
\end{remark}
%%%%%%%%%%%%%%%%%%%%%%%%%%%%%%%%%%%%%
%%%%%%%%%%%%%%%%%%%%%%%%%%%%%%%%%%%%%
%%%%%%%%%%%%%%%%%%%%%%%%%%%%%%%%%%%%%
\section{Appendix}
\label{sec_appendix}

We relegated to this section the proof of 
Ex. \ref{prop_De_n-not-invariant} 
and Thm. \ref{THM:min-invariant}
and \ref{THM:Coherent-choices-Quillen}.
We will start by recalling some results
involving rational homotopy theory,
whose standard reference is
\cite{Felix-Halperin-Thomas_RHT}, mostly to explain the relation between
Quillen minimal models and transferred $A_{\infty}$-coalgebras.

\begin{remark}\label{primerremark} Let us start by noting that $A_\infty$-coalgebra structures on a graded vector space $C$ are in one-to-one correspondence with differentials in the complete tensor algebra $\widehat T(s^{-1}C)=\Pi_{n\ge 1}T^n(s^{-1}C)$ with $T^n(s^{-1}C)=\left(s^{-1}C\right)^{\otimes n}$. 
Here, $s^{-1}C$ denotes the desuspension of $C$, \emph{i.e.,} $(s^{-1}C)_p=C_{p+1}$,
the grading in 
$\widehat T(s^{-1}C)$
is given by
$$\left( \widehat T(s^{-1}C) \right)_p=
\Pi_{n\ge 1}\left( T^n(s^{-1}C) \right)_p,$$
where
$$\left( T^n(s^{-1}C) \right)_p=
\bigoplus_{p_1+\ldots+p_n=p}
\left(s^{-1}C\right)_{p_1} \otimes
\ldots \otimes
\left(s^{-1}C\right)_{p_n}$$
and the product is given by concatenation.

Giving a differential 
$d\colon \widehat T(s^{-1}C) \longrightarrow \widehat T(s^{-1}C)$ 
in this algebra
is equivalent to giving its restriction
$d\colon  s^{-1}C \longrightarrow \widehat T(s^{-1}C)$ 
and extending $d$ as a derivation.
Such
$d\colon  s^{-1}C \longrightarrow \widehat T(s^{-1}C)$ 
can be written as a sum $d=\sum_{n\ge 1}d_n$, with $d_n(s^{-1}C)\subset T^{ n}(s^{-1}C)$, for $n\ge 1$. 
With this notation, 
the operators $\{\Delta_n\}_{n\ge 1}$ and $\{d_n\}_{n\ge 1}$ uniquely determine each other via
\begin{equation}\label{EQ:d_n-VS-Delta_n}
\begin{aligned}
\Delta_n&=-s^{\otimes n}\circ d_n\circ s^{-1}\colon C\to
C^{\otimes n}, \\
d_{n}&=
-(-1)^{\frac{n(n-1)}{2}}(s^{-1})^{\otimes n}\circ\Delta_n\circ s\colon s^{-1}C\to T^{ n}(s^{-1}C).
\end{aligned}
\end{equation}

\end{remark}

\begin{definition}\label{DEF:cobar}
Given an $A_\infty$-coalgebra
$\left( C, \{\De_n\}_n \right)$,
the complete tensor algebra $\widehat T(s^{-1}C)$
with the differential given by 
(\ref{EQ:d_n-VS-Delta_n}) 
forms a DGA
that is called
the \textbf{cobar construction} of 
$\left( C, \{\De_n\}_n \right)$
and is usually denoted by $\Omega C$.
\end{definition}

This construction behaves nicely with respect to morphisms:
giving a morphism of $A_\infty$-coalgebras
$ f\colon (C,\{\De_n\}_{n\ge 1})\to (C',\{\De'_n\}_{n\ge 1})$
is equivalent to giving a morphism of DGAs, 
which we denote in the same way,
between the corresponding cobar constructions
$${f\colon (\widehat T(s^{-1}C),d)\longrightarrow (\widehat T(s^{-1}C'),d').}$$
Indeed, write $f_{|_{s^{-1}C}}=\sum_{k\ge 0} f_k$ with $f_k\colon s^{-1}C\to T^{k}(s^{-1}C')$ and consider the morphisms
$f_{(k)}\colon C\to {C'}^{\otimes k}$ induced by $f_k$  eliminating desuspensions. Then, it is straightforward to check that the equality $fd=d'f$ translates to the identities MI$(i)$ in Definition \ref{DEF:Morph-A-coalg} 
satisfied by the maps $\{f_{(k)}\}_{k\ge 0}$.

\begin{definition}
A \textbf{differential graded Lie algebra} over $\mathbb{Q}$ (DGL henceforth) is a differential graded $\mathbb{Q}$-vector space $(L,\partial)$ in which:

\begin{itemize}

\item[$\bullet$] $L=\oplus_{p\in\bZ}L_p$ in endowed with a linear operation, called Lie bracket,
$$
[\,\,,\,]\colon L_p\otimes L_q\longrightarrow L_{p+q},\qquad p,q\in \bZ,
$$
 satisfying {\em antisymmetry},
$$
[x,y]=(-1)^{|x||y|+1}[y,x],
$$
and the {\em Jacobi identity}
$$
\bigl[x,[y,z]\bigr]=\bigl[[x,y],z\bigr]+(-1)^{|x||y|}\bigl[y,[x,z]\bigr],
$$
for any  homogeneous elements $x,y,z\in L$. In other words, $L$ is a graded Lie algebra.

\item[$\bullet$] The differential $\partial$, of degree $-1$, satisfies the {\em Leibniz rule},
$$
\partial[x,y]=[\partial x,y]+(-1)^{|x|}[x,\partial y],
$$
for any  pair of homogeneous elements $x,y\in L$.
\end{itemize}
\end{definition}

The tensor algebra
$T(V)=\oplus_{n\ge 0} T^n(V)$ generated by the graded vector space $V$ is endowed with a graded Lie algebra structure with the brackets given by commutators:
$$
[a,b]=a\otimes b -(-1)^{|a||b|}b\otimes a,
$$
for any homogeneous elements $a,b\in T(V)$.  Then, the \textbf{free Lie algebra} $\lib(V)$ is the Lie subalgebra of $T(V)$ generated by $V$.

Observe that $\lib (V)$ is filtered as follows,
$$
\lib(V)= \oplus_{n\ge 1}\lib^n(V),
$$
in which $\lib^n(V)$ is the vector space spanned by  Lie brackets of length $n$, that is $\lib^n(V)=\lib(V)\cap T^n(V)$. In particular, to set a differential in $\lib(V)$,
it is enough to define linear maps
$\partial_n\colon V\longrightarrow \lib^n(V)$ for each $ n\ge 1$ in such a way that $\partial=\sum_{n\ge 1}\partial_n$ squares to zero and satisfies the Leibniz rule. Note that $\partial_1$, the  {\em
linear part of} $\partial$, is therefore a differential in $V$.

\begin{remark}\label{RMK:DGL,LV_induces_DGA,TV}
Given a DGL of the form $(\lib(V),\partial)$, the differential $\partial$ can be extended as a derivation of graded algebras to $T(V)$ to turn it into a DGA. In other words, consider the functor
$$
U\colon \catdgl\longrightarrow \catdga
$$
 which associates to every DGL $(L,\partial)$ its \textbf{universal enveloping algebra} $UL$ which is the graded algebra
 $$
 T(L)/\langle [x,y]-(x\otimes y -(-1)^{|x||y|}y\otimes x)\rangle,\quad x,y\in L,
 $$
 with the differential induced by $\partial$. Then, whenever $(L,\partial)=(\lib(V),\partial)$, one has
 $$
 U (\lib(V),\partial)=(T(V),\partial).
 $$
 \end{remark}

In \cite{Quillen69}, D. Quillen associates to every $1$-connected topological space $X$ of the homotopy type of a CW complex, a particular DGL  $(\lib(V),\partial)$ which is \textbf{reduced}, \emph{i.e.,} $V=\oplus_{p\ge 1} V_p$, for which:
\begin{equation}
\label{eq_V_and_ReducedHomology}
H_*(V,\partial_1)\cong s^{-1}\widetilde H_*(X;\bQ)=\widetilde H_{*+1}(X;\bQ).
\end{equation}
\begin{equation}
\label{eq_DGL_and_HomotopyGroups}
H_*(\lib(V),\partial)\cong \pi_*(\Omega X)\otimes \bQ\cong \pi_{*+1}(X)\otimes\bQ.
\end{equation}

This DGL is a \textbf{Quillen model} of $X$  and it is called \textbf{minimal} whenever $\partial_1=0$. In this case, (\ref{eq_V_and_ReducedHomology}) becomes
\begin{equation}
\label{eq_V_and_ReducedHomology_minimal}
V\cong s^{-1}\widetilde H_*(X;\bQ).
\end{equation}

The Quillen minimal model of $X$ is unique up to isomorphism. 
Remarks 
\ref{RMK:DGL,LV_induces_DGA,TV}
and \ref{primerremark}
together with isomorphism (\ref{eq_V_and_ReducedHomology_minimal}) 
assert that 
any Quillen minimal model 
$(\lib(V),\partial)$ of $X$
induces a structure of transferred $A_{\infty}$-coalgebra on
$\widetilde H_*(X;\bQ)$.
%This should give us an idea of the great
%amount of topological information contained in a 
%transferred $A_{\infty}$-coalgebra,
%since two simply-connected CW complexes have isomorphic Quillen minimal models if %and only if they have the same rational homotopy type.
Given a $1$-connected CW complex $Y$, a (not necessarily minimal) Quillen model of $Y$ can be described in terms of a CW decomposition of $Y$ as follows:

\begin{theorem}\label{THM:Quillen-cells}\cite[III.3.(6)]{Tanre83} Let $Y=X\cup_fe^{n+1}$ be a
1-connected space obtained attaching an $(n+1)$-cell to $X$ via a map $f\colon \bS^n\to X$. Let $(\lib(V),\partial)$ be a Quillen model of $X$ and let $\Phi\in \lib(V)_{n-1}$ be a cycle representing the homology class in $H_{n-1}(\lib(V),\partial)$ which is identified, via isomorphism (\ref{eq_DGL_and_HomotopyGroups}) with the homotopy class $[f]\in\pi_n(X)\otimes\bQ$. Then, the free DGL
$
{(\lib(V\oplus \bQ a),\partial')}
$
determined by
$$
\begin{cases}
\partial' v = \partial v, \qquad \text{for all }
v \in V \\
\partial' a=\Phi
\end{cases}
$$
forms a Quillen model of $Y$.
\end{theorem}

We now proceed to prove the claim in
Ex. \ref{prop_De_n-not-invariant}, 
which points out that,
in general, the numbers 
$\dim \Ker {\Delta_n}_{|H_p(X)}$
depend on the choice of transferred $A_\infty$-structure.

\begin{proof}
The space $
X=\mathbb{C} P^2 \vee \mathbb{S}^7,
$ can be CW-decomposed as
% $$ e^0 \cup_0 e^2 \cup_\eta e^4 \cup_0 e^7,$$
$$ e^0 \cup e^2 \cup e^4 \cup e^7,$$
where the attaching map of the 4-cell is the Hopf map $\eta:\mathbb{S}^3 \longrightarrow \mathbb{S}^2$
%\cite[Rmk. 3.4.3 (2)]{Arkowitz11}
and the rest of attaching maps are all trivial.
In turn, $\eta = \frac{1}{2} g \in \pi_3(\mathbb{S}^2) \otimes \mathbb{Q}$, where $g$ denotes the Whitehead product 
$ g=[id_{\bS^2},id_{\bS^2}]. $
% \cite[Rmk. 5.7]{Berglund12}
Hence, via Thm. \ref{THM:Quillen-cells},
the free DGL $ (\lib(x_1,y_3,z_6),\partial) $
forms a Quillen minimal model of $X$, where subscripts denote degree and the differential is given by $\partial x_1 = 0, \partial y_3 = \frac{1}{2} [x_1, x_1]$ and $ \partial z_6 = 0 $.

Consider the isomorphism of (non differential) free Lie algebras
$$
\psi\colon \lib(x_1,y_3,z_6)
\stackrel{\cong}{\longrightarrow} \lib(x_1,u_3,v_6)
$$
given by $\psi(x_1) = x_1, \psi(y_3) = \frac{1}{2} u_3$ and $\psi(z_6) = v_6 - [u_3, u_3].$
Then, set in the free Lie algebra on the right 
the differential
$ \partial=\psi \partial \psi^{-1} $
so that $\psi$ becomes an isomorphism of DGLs between
$ (\lib(V),\partial) = (\lib(x_1,y_3,z_6),\partial) $ and
$ (\lib(W),\partial) = (\lib(x_1, u_3, v_6),\partial).$
A short computation shows that in $(\lib(W),\partial)$, 
$\partial x_1 = 0,
\partial u_3 = [x_1, x_1] $ and
$ \partial v_6 = 2 [[x_1, x_1], u_3]. $
Since $(\lib(V),\partial)$ and $(\lib(W),\partial)$
are isomorphic DGLs,
$(\lib(W),\partial)$
is also a Quillen minimal model of $X$,
and in particular
$
V\cong W \cong s^{-1}\widetilde H_*(X;\bQ).
$
A short computation shows that, in the universal enveloping algebra $(T(W),d)=U(\lib(W),\partial)$,
$
d_3 v_6=4 x_1 \otimes x_1 \otimes u_3 -
4 u_3 \otimes x_1 \otimes x_1.
$
In other words, in the $A_\infty$-coalgebra structure $\{\Delta_n\}_n$ on 
$\widetilde H_*(X;\bQ)$
given by $(\lib(W),\partial)$,
$
\Delta_3 \neq 0.
$

On the other hand, 
in the universal enveloping algebra 
$(T(V),d)=U(\lib(V),\partial)$,
$
d_n = 0
$
for all $n \neq 2$,
which means that
in the $A_\infty$-coalgebra structure $\{\Delta'_n\}_n$ on $\widetilde H_*(X;\bQ)$
given by $(\lib(V),\partial)$,
$
\Delta'_n = 0
$
for all $n \neq 2$;
in particular, $\De_3'=0$.
\end{proof}
	
To prove Thm. \ref{THM:min-invariant}
we will first prove a preliminary result.

\begin{lemma}\label{Lemma:00k00}
Let $\left( C, \{\De_n\}_n \right)$
be an $A_\infty$-coalgebra such that
$\De_n \neq 0$ for some $n\geq 1$.
If we denote by $k$ the lowest integer $n\geq 1$ such
that $\De_n \neq 0$, then
$ \left( C, \{0, \ldots, 0, \De_k, 0, \ldots \} \right) $
forms an $A_\infty$-coalgebra too.
\end{lemma}

\begin{proof}
Let $\left( C, \{\De_n\}_n \right)$ and $k$
be as in the statement and 
define
$\De_n'\coloneqq
\begin{cases}
0, & n\neq k\\
\De_k, & n = k.
\end{cases}
$

Let us denote by $\{\text{SI}(n)\}_{n}$
the Stasheff identities (Definition 
\ref{DEF:A-infty-coalgebra}) on
$\{\De_n\}_n$ and by 
$\{\text{SI}'(n)\}_{n}$
the Stasheff identities on
$\{\De'_n\}_n$.
The pair
$\left( C, \{\De'_n\}_n \right)$ forms an $A_\infty$-coalgebra
if and only if the identities $\text{SI}'(n)$ hold for all $n\geq 1$.
Notice that all of them except for 
$\text{SI}'(2k-1)$
become the trivial identity 
$0 = 0,$
and that $\text{SI}'(2k-1)$
becomes 
$$ \sum_{j=0}^{k+1} (-1)^{k+j+kj}
\left( 1^{\otimes k-1-j} \otimes \De_k \otimes
1^j \right) \De_k = 0,$$
since this is the only 
$\text{SI}'(n)$
in which $\De_k$
composes with 
$\De_k$ (tensored by identities),
which is the only non-zero operation in $\{\De_n'\}$.

Let us now look at the identity 
$\text{SI}(2k-1)$ on $\{\De_n\}_n$.
If some $\De_n$ with $n>k$ appears in 
$\text{SI}(2k-1)$,
then $\De_n$ (or $\De_n$
tensored by identities)
is pre or post composed with
$\De_m$ (or with $\De_m$
tensored by identities) for some $m<k$.
Since $\De_m = 0$, all such compositions vanish.
Therefore, 
$\text{SI}(2k-1)$
coincides with 
$\text{SI}'(2k-1)$.
Since we know that $\left( C, \{\De_n\}_n \right)$
forms an $A_\infty$-coalgebra
and hence that
$\text{SI}(2k-1)$ holds,
it follows that 
$\text{SI}'(2k-1)$
holds too and we have therefore checked that
$\left( C, \{\De'_n\}_n \right)$
satisfies all Stasheff identities.
\end{proof}

Next we prove Thm. \ref{THM:min-invariant}.
In this proof we will make use of the algebraic Milnor-Moore
spectral sequence.
The background on spectral sequences
needed to follow this proof can be found in 
\cite[\S 18 \& \S 23(b)]{Felix-Halperin-Thomas_RHT}
and
\cite[Ch. 1]{Greub-Halperin-Vanstone76}.

\begin{proof}
Let $C$ be a graded vector space with $C_p=0$ for all $p<0$.
Let 
$(C,\{\De_n\}_n)$
be a minimal 
$A_\infty$-coalgebra
and let
$ \Omega C = (\widehat T(s^{-1}C),d) $
be its cobar construction
(Definition \ref{DEF:cobar}). We can then write 
$d$ as the sum
$\sum_{n\geq 1} d_n,$
where each
$$ d_n\colon
\widehat T(s^{-1}C)
\longrightarrow T^{\geq n} (s^{-1}C) = 
\Pi_{p\ge n} T^p(s^{-1}C) $$
satisfies
$$ d_n \left( T^p(s^{-1}C) \right) \subseteq
T^{p+n-1}(s^{-1}C) $$
for all $p\in\N$
and acts on $T^1(s^{-1}C) = s^{-1}C$
as the composition
$$ \xymatrix{
d_n\colon s^{-1}C \ar[r]^-{s} &
C \ar[rrr]^-{-(-1)^{\frac{n(n-1)}{2}} \Delta_n} &&&
C^{\otimes n} \ar[rr]^-{{(s^{-1})}^{\otimes n}} &&
T^n(s^{-1}C).
}$$

Let us set
$$
k \coloneqq
\begin{cases}
 \min \{ n \,|\, \Delta_n \neq 0 \} , \qquad \text{if } \{ n \,|\, \Delta_n \neq 0 \} \neq \emptyset \\
+\infty, \qquad \text{otherwise,}
\end{cases}
$$
and let us define $\Omega' C$, $d' = \sum_{n\geq 1} d'_n$ and $k'$ analogously for a minimal $A_\infty$-coalgebra
$(C,\{\De'_n\}_n)$
isomorphic to 
$(C,\{\De_n\}_n)$.
First thing to notice is that since the $A_\infty$-algebras are isomorphic, their cobar constructions $\Omega C$ and $\Omega' C$
are isomorphic DGAs.

Let us now divide the proof into two complementary cases:

\noindent \emph{$\cdot$ Case $\Delta_n=0$ for all $n$:}

If $\Delta_n=0$ for all $n$, $d$ must be 0. The fact that $\Omega C$ and $\Omega' C$ are DGA isomorphic forces $d'$ to be 0 too. This implies that $\Delta'_n=0$ for all $n$.
Hence, $k = +\infty = k'$ and $ \dim \Ker {\De_m}_{|C_p}
= \dim C_p = \dim \Ker {\De'_m}_{|C_p} $ holds for every $m\geq 1$ and $p\geq 0$.

\noindent \emph{$\cdot$ Case $\Delta_n\neq0$ for some $n$:}
 
In this case, $k$ becomes the integer
$\min \{ n \,|\, \Delta_n \neq 0 \}. $
Since $(C,\{\De_n\}_n)$ is minimal, $k$ must be at least 2.
Note that $k$ is obviously equal to
$ \min \{ n \,|\, d_n \neq 0 \}. $
Applying again the argument just used in the previous case, if there exists some $n$ such that $\Delta_n\neq0$, there must exist some $n'$ such that $\Delta'_{n'}\neq0$ as well. Hence, all this holds for $k'$ as well.
Most of the rest of the proof will consist on showing that $k = k'$ and that the DGAs
$(\widehat T(s^{-1}C),d_k)$ and
$(\widehat T(s^{-1}C),d'_k)$
are isomorphic.

Assume $k\le k'$ and let
$ \widehat T(s^{-1}C)=F^0 \supseteq F^1
\supseteq \ldots \supseteq
F^p \supseteq F^{p+1}
\supseteq \ldots $
be the \textbf{word-length filtration}
of $\Omega C$,
\emph{i.e.,}
for all $p\in\N$,
$F^p$ is
the differential ideal given by
$$F^p = T^{\geq p}(s^{-1}C).$$
In the algebraic Milnor-Moore
spectral sequence,
$E_0$ is the graded algebra associated
to this filtration,
\emph{i.e.,}
for all $p\in\N$,
$$ E_0^p = \frac{F^p}{F^{p+1}}
\cong T^p(s^{-1}C)
\hspace{10mm}
\text{and}
\hspace{10mm}
E_0 = \bigoplus_{p\in\N} E_0^p 
\cong T(s^{-1}C),$$
and the differential 
$\dlie_0=\bigoplus_{p\in\N} \dlie_0^p\colon  E_0 \longrightarrow E_0 $
is the map induced by $d$.
Thus, 
for all $p\in\N$,
$ \dlie_0^p\colon E_0^p \longrightarrow E_0^p $
must be
$ d_1\colon  T^p(s^{-1}C) \longrightarrow T^p(s^{-1}C),$
which is 0, since $(C,\{\De_n\}_n)$ is minimal.
Hence, $ (E_0, \dlie_0) \cong 
(\widehat T(s^{-1}C),0). $

In this spectral sequence,
for all $p\in\N$, we have
$ E_1^p \cong E_0^p \cong T^p(s^{-1}C), $
and the map
$ \dlie_1^p\colon E_1^p \longrightarrow E_1^{p+1} $
induced by $d$
is thus
$ d_2\colon  T^p(s^{-1}C) \longrightarrow T^{p+1}(s^{-1}C).$
Hence, 
$ (E_1, \dlie_1) \cong (\widehat T(s^{-1}C),d_2). $
Similarly, we can easily verify that
$$ (E_0, \dlie_0) \cong \ldots \cong
(E_{k-2}, \dlie_{k-2}) \cong 
(\widehat T(s^{-1}C),0)
\hspace{10mm}
\text{and}
\hspace{10mm}
(E_{k-1}, \dlie_{k-1}) \cong
(\widehat T(s^{-1}C),d_k). $$

Let $\varphi\colon  \Omega C \longrightarrow
\Omega C'$
be an isomorphism of DGAs, that exists 
because $(C,\{\De_n\}_n)$ and 
$(C,\{\De'_n\}_n)$ are isomorphic
$A_\infty$-coalgebras.
Since $\varphi$ is filtered, 
it induces a morphism 
of spectral sequences,
$\{ E_r(\varphi) \}_r $,
from the algebraic Milnor-Moore
spectral sequence associated to
$\Omega C$ to that associated
to $\Omega' C$.
Recall that 
the linear part of $\varphi$,
$\varphi_1\colon s^{-1}C
\longrightarrow s^{-1}C,$
is defined by $\varphi(s^{-1}c)=\varphi_1(s^{-1}c)+\Phi$,
where $\Phi\in T^{\ge 2}(s^{-1}C)$. 
Since $\varphi$ is an isomorphism,
$\varphi_1$
must also be an isomorphism as well.
Now observe that $E_0(\varphi)=E_{k-2}(\varphi)$ is
precisely the isomorphism 
$\widehat T(\varphi_1)\colon 
(\widehat T(s^{-1}C),0)\longrightarrow
(\widehat T(s^{-1}C),0).$
Applying a version of the Comparison Theorem
\cite[Thm I in \S 1.3]{Greub-Halperin-Vanstone76}
in terms of DGAs, we conclude that
$ E_{k-1}(\varphi)\colon  
\left( T(s^{-1}C),d_{k} \right)
\longrightarrow
\left( T(s^{-1}C),d'_{k} \right)
$
is an isomorphism of DGAs.
In particular, $d'_k$ cannot be zero.
Hence $k=k'$ and the first claim in Thm.
\ref{THM:min-invariant} holds.

Lemma \ref{Lemma:00k00} guarantees that
$$ (C,\{0,\ldots,0,\De_k,0,\ldots\})
\hspace{10mm}
\text{and}
\hspace{10mm}
(C,\{0,\ldots,0,\De'_k,0,\ldots\}) $$
form two $A_\infty$-coalgebras.
To prove that they are isomorphic,
simply observe that the corresponding cobar constructions are $(\widehat T(s^{-1}C),d_k)$ and $(\widehat T(s^{-1}C),d'_k),$ which we have just seen are isomorphic DGAs via
 $E_{k-1}(\varphi)$.
 This tells us there exists an
isomorphism of $A_\infty$-coalgebras
$$ f \colon (C,\{0,\ldots,0,\De_k,0,\ldots\})
\longrightarrow
(C,\{0,\ldots,0,\De'_k,0,\ldots\}). $$
The identity MI$(k)$
from Definition \ref{DEF:Morph-A-coalg}
becomes
$ \De_k' f_{(1)} =
{f_{(1)}}^{\otimes k} \De_k, $
and since $f_{(1)}$ is an isomorphism, 
it follows that 
$ \dim \Ker {\De_k}_{|C_p}
=\dim \Ker {\De'_k}_{|C_p}, $
for all $p\geq 0$.

Obviously, by the definition of $k$, $ \dim \Ker {\De_m}_{|C_p}
= \dim C_p = \dim \Ker {\De'_m}_{|C_p} $ holds too for every $m<k$ and $p\geq 0$.
\end{proof}

We now prove Thm. \ref{THM:Coherent-choices-Quillen}, 
which states the possibility, under certain assumptions, of finding compatible
$A_\infty$-structures yielding to persistence modules
with which to compute $A_\infty$-persistence.

\begin{proof}
Let $K_0$ be a 1-connected CW complex.
If for all $0<i\leq N$, $K_i$ denotes the wedge $K_{i-1} \vee \mathbb{S}^{n_i}$ for some $n_i > 1$,
then:
\begin{itemize}
\item each $K_i$ is a 1-connected CW complex as well, and 
\item the maps $f_p^{i,i+1}: \widetilde H_p(K_i; \bQ) \longrightarrow \widetilde H_p(K_{i+1}; \bQ)$
and $f^{i,i+1}: \widetilde H_*(K_i; \bQ) \longrightarrow \widetilde H_*(K_{i+1}; \bQ)$ are inclusions.
\end{itemize} 

We start the construction by choosing any Quillen
minimal model $(\lib(V),\partial)$ of $K_0$.
Since $K_1$ can be decomposed as $K_1=K_0\cup_f e^{n_i}$, where the attaching map $f\colon \bS^{n_i}\to K_0$ is the trivial one, Thm. \ref{THM:Quillen-cells} then tells us that the free DGL
${(\lib(V\oplus \bQ a),\partial')}$
determined by
$$
\begin{cases}
\partial' v = \partial v, \qquad \text{for all }
v \in V \\
\partial' a=0 \\
|a|=n_i-1
\end{cases}
$$
is a Quillen model of $K_1.$
Furthermore, since
${\partial'_1}_{|V}=\partial_1 =0$
and 
$\partial' a=0$,
we have that $\partial'_1=0$,
so the model
$(\lib(V\oplus \bQ a),\partial')$ is indeed minimal.

Applying this argument at each step, 
we obtain a Quillen minimal model 
$
(\lib(V_i),\partial^i)
$
for each $K_i$
so that
$
\partial^{i+1} v = \partial^i v
$
for any $v \in V_i \subset V_{i+1}$
and for all $0 \leq i < N$.
These Quillen minimal models then translate into
transferred $A_\infty$-coalgebras
$\left( \widetilde H_*(K_i; \bQ), \{ \De_n^i \}_n
\right)$
such that for all $n\geq 1$, $p\geq 0$
and $0 \leq i < N$,
if $\Delta^i$ denotes the restriction of ${\Delta_n^i}$ to 
${\widetilde H_p(K_i; \bQ)}$,
then
$$
\De^{i+1} \alpha =
\De^i\alpha
$$
for every $\alpha \in \widetilde H_p(K_i; \bQ) \subset 
\widetilde H_p(K_{i+1}; \bQ)$.
Since the maps
$f_p^{i,i+1}$ and $f^{i,i+1}$ are inclusions,
this means that the equality
$$
\De^{i+1} f_p^{i,i+1} =
\left( f^{i,i+1} \right)^{\otimes n}
\De^i
$$
holds
and hence so does
$f_p^{i,i+1} \left( \Ker \Delta^i \right)
\subseteq \Ker \Delta^{i+1}.$ 
\end{proof}
%%%%%%%%%%%%%%%%%%%%%%%%%%%%%%%%%%%%%
%%%%%%%%%%%%%%%%%%%%%%%%%%%%%%%%%%%%%
%%%%%%%%%%%%%%%%%%%%%%%%%%%%%%%%%%%%%
\section{Conclusions}
\label{sec_conclusions}

Let us use the diagram in Fig. \ref{fig:DiagramaResumen}
to enumerate the contributions of this work.
 In relation to Fig. \ref{fig:DiagramaResumen}a,
 we provide a new proof of the Fundamental Theorem of Persistent Homology
 (see Lemma \ref{lemma_persModsSplit} and
 Thm. \ref{Thm:Fundamental_thm_PH}).
 Our work on Fig. \ref{fig:DiagramaResumen}b
 consists of showing how the part of $A_\infty$-structures 
 we focus on in $A_\infty$-persistence can form powerful descriptors
 (Thm. \ref{THM:min-invariant}, 
Cor. \ref{COR:min-invariant},
Prop. \ref{Prop:Borromean-m3} and \ref{Prop:ExtendingBorromean}
and Ex. \ref{prop_De_n-not-invariant}).
In relation to Fig. \ref{fig:DiagramaResumen}c,
we study $A_\infty$-persistence decompositions
 of persistence modules (Thm.
 \ref{Thm_Fund_Thm_A_infty-p_Functorial} and
\ref{THM:Coherent-choices-Quillen}).

In the current and following papers we keep developing the theory
of $A_\infty$-persistence because of its exciting great potential:
persistent homology has been used successfully in many areas, such as
digital imaging, sensor networks coverage, materials science, molecular modelling, signal processing, virus evolution and diagnosis of hepatic lesions, and all this has been possible by using persistence at the level of Betti numbers.
Hence, by enhancing the power of persistence through 
the use of $A_\infty$-structures, who knows where we can get?

% Same with references:
%In the current and following papers we further develop the theory
%of $A_\infty$-persistence, which we find exciting because of its great potential:
%persistent homology has been used successfully in many areas, such as
%digital imaging \cite{Robins-Wood-Sheppard11, 
%Perea-Carlsson14}, sensor networks coverage
%\cite{de_Silva-Ghrist07}, materials science
%\cite{Krama-MischaikowEtAl14, MacPherson-Schweinhart12}, molecular modelling
%\cite{Agarwal-Edelsbrunner-Harer-Wang06, Kovacev-Nikolic-Bubenik-Nikolic-Heo16, Gameiro-Hiraoka-Izumi-Kramar_Mischaikow-Nanda15}, signal processing
%\cite{Emrani-Gentimis-Krim14, Brown-Knudson09}, virus evolution
%\cite{Chan-Carlsson-Rabadan13} and diagnosis of hepatic lesions
%\cite{Adcock-Carlsson-Rubin14}, and all this has been possible by using persistence
%that computes topological information at the level of Betti numbers.
%Hence, by enhancing the power of persistence through 
%the use of $A_\infty$-structures, who knows where we can get?

%%%%%%%%%%%%%%%%%%%%%%%%%%%%%%%%%%%%%
%%%%%%%%%%%%%%%%%%%%%%%%%%%%%%%%%%%%%
%%%%%%%%%%%%%%%%%%%%%%%%%%%%%%%%%%%%%
\section*{Acknowledgements}
I would like to thank Prof. Aniceto Murillo and Prof. Jim Stasheff for their valuable feedback on this work.

%%%%%%%%%%%%%%%%%%%%%%%%%%%%%%%%%%%%%
%%%%%%%%%%%%%%%%%%%%%%%%%%%%%%%%%%%%%
%%%%%%%%%%%%%%%%%%%%%%%%%%%%%%%%%%%%%
% 1st run Quick Build with this:
%\bibliographystyle{plain}
%{\footnotesize
%%\bibliography{/Users/Kiko/Dropbox/BibTeX/mybib_Kiko}{}}
%}
%
% This will create a file 'optimisingAinftyPersistGps.bbl'
% 
% Now you can substitute this BibTex section by
% \input{optimisingAinftyPersistGps.bbl}

\end{document}